\documentclass{amsart}[12pt]
\usepackage{amsmath,amssymb,amsthm,amsfonts}
\usepackage{paralist}
\usepackage{booktabs}
\usepackage[usenames]{color}
\usepackage{comment}
\usepackage{graphicx}
\usepackage{graphics}
\usepackage{epsfig}
\usepackage{epstopdf}
\usepackage{psfrag}
\usepackage{latexsym}
\usepackage[shortlabels]{enumitem}
\usepackage{thmtools}
\usepackage{url}
\usepackage[all,cmtip]{xy}
\usepackage{tikz}
\tikzset{font={\fontsize{4pt}{12}\selectfont}}
\usepackage{hyperref} %hyperref wants to be loaded last

\usepackage[enableskew]{youngtab}
\usepackage{ytableau}

\newcommand{\Alt}{\mathsf{Alt}}

\newcommand{\ssf}{\mathsf{f}}

\title[Standard Young tableaux of skew shape]
{Asymptotics of the number of standard Young \\ tableaux of skew shape}

\author[Alejandro Morales, Igor Pak and Greta Panova]{Alejandro H.~Morales$^\star$,
\ \ Igor Pak$^\star$ \ \ and \ \ Greta Panova$^\dagger$}

\thanks{\today}
\thanks{\thinspace ${\hspace{-.45ex}}^\star$Department of Mathematics,
UCLA, Los Angeles, CA~90095.
\hskip.06cm
Email:
\hskip.06cm
\texttt{\{ahmorales,\ts{pak}\}@math.ucla.edu}}

\thanks{\thinspace ${\hspace{-.45ex}}^\dagger$Department of Mathematics,
 UPenn, Philadelphia, PA~19104.
\hskip.06cm
Email:
\hskip.06cm
\texttt{panova@math.upenn.edu}}

\newcommand{\ED}{\mathcal{E}}
\newcommand{\ed}{e}

\newcommand{\SSYT}{\operatorname{SSYT}}
\newcommand{\SYT}{\operatorname{SYT}}

\DeclareMathOperator{\area}{area}

\DeclareMathOperator{\width}{width}
\DeclareMathOperator{\depth}{depth}

\declaretheorem[numberwithin=section]{theorem}
\declaretheorem[numberlike=theorem]{lemma}
\declaretheorem[numberlike=theorem]{proposition}
\declaretheorem[numberlike=theorem]{corollary}

\declaretheorem[numberlike=theorem, style=definition]{remark}

\declaretheorem[numberlike=theorem, style=definition]{example}

\numberwithin{equation}{section} % requires package amsthm

\def\emp{\varnothing}

\def\sqsm{{\text{\small $\boxplus$}}}

\def\nn{\mathbb N}

\def\rr{\mathbb R}

\def\ov{\overline}

\def\De{\Delta}

\def\la{\lambda}
\def\ga{\gamma}
\def\si{\sigma}
\def\de{\delta}

\def\al{\alpha}
\def\be{\beta}
\def\om{\omega}

\def\vk{\varkappa}

\def\cC{\mathcal C}

\def\cp{\mathcal P}
\def\cP{\mathcal P}
\def\cq{\mathcal Q}

\def\ssu{\subset}

\def\<{\langle}
\def\>{\rangle}
\def\La{\Lambda}

\def\ups{\upsilon}
\def\vt{\vartheta}

\def\0{{\mathbf 0}}

\def\SS{{S}}
\def\bal{\ov{\al}}
\def\bbe{\ov{\be}}
\def\hba{\hslash}

\def\HH{\Phi}
\def\SP{{\textsc{\#P}}}

\def\.{\hskip.06cm}
\def\ts{\hskip.03cm}

\def\ed{ {\xi} }
\def\lan{ {\textsc{M}} }
\def\lc{ {\textsc{L}} }

\DeclareMathOperator{\br}{\textsc{b}}

\begin{document}

\begin{abstract}
We give new bounds and asymptotic estimates on the number of
standard Young tableaux of skew shape in a variety of special
cases.  Our approach is based on Naruse's hook-length formula.
% We apply our bounds to \emph{thick ribbons} $\de_{k+r}/\de_k$.
We also compare our bounds with the existing bounds on the numbers
of linear extensions of the corresponding posets.
\end{abstract}

\ytableausetup{smalltableaux}

\maketitle
%\tableofcontents

%----------------------------------------------------------------
\section{Introduction} \label{sec:intro}
%----------------------------------------------------------------

The classical {\em hook-length formula} (HLF) allows one to compute
the number $f^\la = |\SYT(\la)|$ of standard Young tableaux of
a given shape~\cite{FRT}.  This formula had profound applications in Enumerative and
Algebraic Combinatorics, Discrete Probability, Representation Theory
and other fields (see e.g.~\cite{Rom,Sag,EC2}).  Specifically, the HLF
allows to derive asymptotics for $f^\la$ for various families of
``large'' partitions~$\la$.  This was famously used to compute the diagram of a
random representation of $S_n$ with respect to the Plancherel measure $(f^\la)^2/n!$,
see~\cite{LS,VK} (see also~\cite{Bia,Stanley_ICM}).

For skew shapes $\la/\mu$, little is known about the asymptotics,
since there is no multiplicative formula for $f^{\la/\mu}=|\SYT(\la/\mu)|$.
For large $n=|\la/\mu|$, the asymptotics are known for
few special families of skew shapes (see~\cite{AR}), and for fixed~$\mu$
(see~\cite{OO,RV,Stanley_skewSYT}).  In this paper we show that the ``naive HLF''
give good approximations for $f^{\la/\mu}$ in many special cases of interest.

\smallskip

Formally, let $\lambda$ be a partition of~$n$.  Denote by
$f^\la = |\SYT(\la)|$ the number of standard Young tableaux of shape~$\la$.
We have:
\begin{equation} \label{eq:hlf} \tag{HLF}
f^{\lambda} \, = \, \frac{n!}{\prod_{u\in \lambda} \. h(u)}\,,
\end{equation}
where $h(u)=\lambda_i-i+\lambda'_j-j+1$ is the {\em hook-length} of the
square $u=(i,j)$.
Now, let $\la/\mu$ be a skew shape, $n=|\la/\mu|$.  By analogy with the HLF, define
\begin{equation} \label{eq:na-hlf} \tag{naive HLF}
F(\la/\mu) \, := \, \frac{n!}{\prod_{u\in \lambda/\mu} \. h(u)}\,,
\end{equation}
where $h(u)$ are the usual hook-lengths in~$\la$.\footnote{Note
that $F(\la/\mu)$ is not necessarily an integer.}

Our main technical tool is Theorem~\ref{t:main}, which gives
$$(\ast) \qquad \ \
F(\la/\mu) \, \leq \, f^{\lambda/\mu} \, \leq \, \ed(\lambda/\mu) \. F(\la/\mu)\,,
$$
where $\ed(\lambda/\mu)$ is defined in Section~\ref{s:def}.
These bounds turn out to give surprisingly sharp estimates
for $f^{\lambda/\mu}$, compared to standard bounds on the
number~$e(\cP)$ of linear extensions for general posets.\footnote{Both
$e(\cP)$ and $f^{\lambda/\mu}$ are standard notation in respective
areas.  For the sake of clarify and to streamline the notation, we use $e(\la/\mu)=|\SYT(\la/\mu)|$
throughout the paper (except for the Introduction and Final Remarks sections~\ref{sec:intro} and~\ref{s:fin}). }
We also give several examples when the lower bound is sharp
but not the upper bound, and vice versa
(see e.g.~$\S$\ref{ss:ex-zigzag} and~$\S$\ref{ss:ex-slim-OO}).

Let us emphasize an important special case of \emph{thick ribbons}
$\de_{k+r}/\de_k$, where
$\delta_k =(k-1,k-2,\ldots,2,1)$ denotes the \emph{staircase shape}.
% (see Figure~\ref{f:stable}).
The following result illustrates the
strength our bounds (cf.~$\S$\ref{ss:posets-compare}).

\begin{theorem} \label{t:ribbons}
Let $\ups_k=(\de_{2k}/\de_k)$, where $\de_k=(k-1,k-2,\ldots,2,1)$.
Then
$$\frac16 \ts - \ts \frac{3}{2}\ts\log 2 \ts + \ts \frac{1}{2}\ts\log 3 \. + \. o(1) \,  \le \,
\frac{1}{n} \left(\log f^{\ups_k} \. - \. \frac12 \ts\ts n \log n\right) \,  \le \,
\frac16 \ts - \ts \frac{7}{2}\ts\log 2 \ts + \ts 2\ts\log 3 \.+ \. o(1)\ts,
$$
where $\ts n=|\ups_k|=k(3k-1)/2$.
\end{theorem}

Here the LHS $\approx -0.3237$, and the RHS $\approx -0.0621$.
Note that the number $f^{\de_{k+r}/\de_k}$ of standard Young tableaux
for thick ribbons (see \cite[\href{https://oeis.org/A278289}{A278289}]{OEIS}) have been previously considered in~\cite{BR},
but the tools in that paper apply only for $r\to \infty$.

We should mention that in the theorem and in many other special cases,
the leading terms of the asymptotics are easy to find.  Thus most of our
effort is over the lower order terms which are much harder to determine
(see Section~\ref{s:fin}).  In fact, it is the lower order terms that
are useful for applications (see $\S$\ref{ss:ex-LR} and~$\S$\ref{ss:fin-lower}).

\smallskip

The rest of the paper are structured as follows.  We start with
general results on linear extensions (Section~\ref{s:posets}),
standard Young tableaux of skew shape (Section~\ref{s:def}),
and excited diagrams (Section~\ref{s:lemmas}).  We then
proceed to
% upper bounds on $\ed(\la/\mu)$, the number of excited diagrams involved in NHLF, in Section~\ref{s:lemmas}.
%
our main results concerning asymptotics for $f^{\la/\mu}$ in the following
cases:

\smallskip

\hskip.5cm $(1)$ \ when both $\la,\mu$ have the \emph{Thoma--Vershik--Kerov limit} (Section~\ref{s:vk}).
Here the Frobenius

\hskip1.2cm
coordinates scale linearly and $f^{\la/\mu}$ grow exponentially.

\smallskip

\hskip.5cm $(2)$ \ when both $\la,\mu$ have the \emph{stable shape limit} (Section~\ref{s:stable}).
Here the row and column lengths

\hskip1.2cm
scale as $\sqrt{n}$, and $f^{\la/\mu} \approx \sqrt{n!}$
up to an exponential factor.

\smallskip

\hskip.5cm $(3)$ \ when $\la/\mu$ have \emph{subpolynomial depth} (Section~\ref{s:sub-poly}).
Here both the row and column lengths

\hskip1.2cm
of $\la/\mu$ grow as $n^{o(1)}$, and
$f^{\la/\mu} \approx n!$ up to
a factor of \emph{intermediate growth} (i.e.~super-

\hskip1.2cm
exponential and subfactorial), which can be determined by the depth growth function.
\smallskip

\hskip.5cm $(4)$ \ when $\la/\mu$ is a large ribbon hook (Section~\ref{s:ribbon}).
Here $\la/\mu$ scale linearly along fixed curve,

\hskip1.2cm
 and
$f^{\la/\mu} \approx n!$ up to an exponential factor.
\smallskip

\hskip.5cm $(5)$ \ when $\la/\mu$ is a \emph{slim shape} (Section~\ref{s:slim}).
Here $\mu$ is fixed and $\ell, \la_\ell/\ell\to \infty$, where $\ell=\ell(\la)$.

\hskip1.2cm
Here $f^{\la/\mu} \sim f^\la \ts f^\mu/|\mu|!$.

\smallskip
\noindent
We illustrate these cases with various examples.
Further examples and more specialized applications are given in
sections~\ref{s:ribbons} and~\ref{s:ex}.  We conclude with final
remarks in Section~\ref{s:fin}.

\bigskip

\section{Linear extensions of posets}\label{s:posets}

\subsection{Notation}\label{ss:posets-not}
We assume the reader is familiar with standard definitions and notation
of Young diagrams, Young tableaux, ranked posets, linear extensions,
chains, antichains, etc.
In case of confusion, we refer the reader to~\cite{EC2}, and
will try to clarify the notation throughout the paper.

To further simplify the notation, we use the same letter
to denote the partition and the corresponding Young diagram.  To avoid
the ambiguity, unless explicitly stated otherwise, we always assume
that skew partitions are connected. To describe disconnected shapes,
we use $\ts \la/\mu \circ \pi/\tau \ts$ notation.

We make heavy use of Stirling's formula $\ts \log n! = n \log n - n  +O(\log n)$.
% In fact, we need precise bounds for all $n\ge 1$~:
% $$\log n! \. - \. n \ts \log n +\. n \. -\frac12 \log n \, \in \, [ \. \log\sqrt{2\pi}\, , \, 1 \.] \ts.$$
%
%$$\log n! \, = \,  n \ts \log n \. - \. n \. + \. O(\log n) \
%\quad \text{as} \ \ n\to \infty\ts.
%$$
Here and everywhere below $\ts\log\ts$ denotes natural logarithm.
There is similar formula for the \emph{double factorial}  \ts
$(2n-1)!! = 1\cdot 3 \cdot 5 \cdots (2n-1)$, the
\emph{superfactorial}  \ts $\Phi(n) =  1!\cdot 2! \ts \cdots \ts n!$,
the \emph{double superfactorial} \ts $\Psi(n) =  1!\cdot 3! \cdot 5! \ts \cdots \ts (2n-1)!$,
and the \emph{super doublefactorial} \ts $\La(n) =  1!!\cdot 3!! \cdot 5!! \ts \cdots \ts (2n-1)!!$~:
$$
\aligned
& \log \ts (2n-1)!! \,  = \,  n \ts \log n \. + \. (\log 2 \ts - \ts 1)\ts n \. + \. O(1)\ts, \\
& \log \Phi(n) \,  = \,  \frac{1}{2} \. n^2 \ts \log n \. - \. \frac34 \. n^2 \. +
\. 2\ts n \ts \log n \. + \. O(n)\ts,\\
& \log \Psi(n) \,  = \,  n^2 \ts \log n \. + \. \left(\log 2 \ts - \ts \frac32\right) n^2 \. +
\. \frac52 \ts n \ts \log n \. + \. O(n)\ts, \\
& \log \La(n) \,  = \,  \frac12\ts n^2 \ts \log n \. + \. \left(\frac{\log 2}{2} \ts - \ts \frac34\right) n^2 \. +
\. \frac12 \ts n \ts \log n \. + \. O(n)
\endaligned
$$
(see \cite[\href{http://oeis.org/A001147}{A001147}]{OEIS}, \cite[\href{http://oeis.org/A008793}{A008793}]{OEIS}, \cite[\href{http://oeis.org/A168467}{A168467}]{OEIS}), and \cite[\href{http://oeis.org/A057863}{A057863}]{OEIS}).

Finally, we use the standard asymptotics notations $f\sim g$, $f=o(g)$, $f=O(g)$ and $f=\Omega(g)$,
see e.g.~\cite[$\S$A.2]{FS}.  For functions, we use $f\approx g$ to denote $\ts \log f \sim \log g$,
see the introduction.  For constants, we use $c\approx c'$ to approximate their numerical value
with the usual rounding rules, e.g. $\pi \approx 3.14$ and $\pi \approx 3.1416$.

\subsection{Ranked posets}\label{ss:posets-le}
Let $\cP$ be a ranked poset on a finite set $X$ with linear ordering
denoted by~$\prec$.  Unless stated otherwise, we assume that $|X|=n$.
Let $e(\cP)$ be the number of \emph{linear extensions} of~$\cP$.

Denote by $\ell=\lc(\cP)$ and $m=\lan(\cP)$ the length of the \emph{longest
chain} and the \emph{longest antichain}, respectively.  Let $r_1,\ldots,r_k$
denote the number of of elements in $X$ of each rank, so $k \ge \ell$ and
$n=r_1+\ldots+r_k$.  Similarly, let $\cP$ have a decomposition into chains
$C_1,\ldots,C_m$,
and denote $\ell_i=|C_i|$, so $\ell_1+\ldots+\ell_m=n$. Recall that such
decompositions exists by the \emph{Dilworth theorem} (see e.g.~\cite{Tro}).

\begin{theorem} \label{t:poset-gen}
For every ranked poset $\cP$ as above, we have:
$$
r_1! \cdots r_k! \, \le \, e(\cP) \, \le \, \frac{n!}{\ell_1!\ts \cdots \ts\ell_m!}\..
$$
\end{theorem}

These bounds are probably folklore; for the lower bound see e.g.~\cite{Bri}.
They are easy to derive but surprisingly powerful.  We include a quick
proof for completeness.

\begin{proof}
For the lower bound, label elements of rank~1 with numbers $1,\ldots,r_1$
in any order, elements of rank~2 with numbers $r_1+1,\ldots,r_1+r_2$
in any order, etc.  All these labelings are clearly linear extensions and
the bound follows. \ts
For the upper bound, observe that every linear extension of $\cP$ when restricted
to chains, defines an ordered set-partition of $\{1,\ldots,n\}$ into $m$~subsets of sizes
$|C_1|,\ldots,|C_m|$.  Since this map is an injection, this implies the upper bound.
\end{proof}

\begin{remark}
Note that the lower bound in the theorem extends to more general antichain
decompositions which respect the ordering of~$\cP$. Although in some
cases this can lead to small improvements in the lower bounds, for
applications we consider the version in the theorem suffices.
\ts It is also worth noting that the upper bound in Theorem~\ref{t:poset-gen}
 is always better than the easy to use upper bound
$e(\cP)\le m^n$ (cf.~\cite{Bri,BT}).
\end{remark}

Let us also mention the following unusual bound for the number of
linear extensions of general posets.  Denote also by
$\br(x)=\#\{y\in \cP, y\succcurlyeq x\}$
the size of the upper ideal in $\cP$ spanned by~$x$.

\begin{theorem} [\cite{HP}]  \label{t:poset-ineq}
For every poset $\cP$, in the notation above, we have:
$$
e(\cP) \, \ge \, \frac{n!}{\prod_{x\in \cP} \.\br(x)}\,.
$$
\end{theorem}

The lower bound was proposed by Stanley~\cite[Exc.~3.57]{EC2}
and proved by Hammett and Pittel~\cite{HP}.
It is tight for forests, i.e.\ disjoint unions
of tree posets (see e.g.~\cite{Sag}).
Note that this bound can be different for the poset $\cP$ and
the \emph{dual poset}~$\cp^\ast$.
We refer to $\S$\ref{ss:fin-linear-extensions} further references
on the number of linear extensions.
% Below we give examples illustrating both theorems.

\medskip

\subsection{Square shape}\label{ss:posets-ex}
The following is the motivating example for this paper.
Let $\la=(k^k)$, $\mu = \emp$, $n=k^2$.
Clearly, $e(k^k) = \bigl|\SYT(\la)\bigr|$
(see \cite[\href{http://oeis.org/A039622}{A039622}]{OEIS}).

Observe that $(r_1,r_2,\ldots) = (1,2,\ldots, k-1,k,k-1,\ldots,1)$
in this case and $m=\lan(k^k)=k$.  Theorem~\ref{t:poset-gen} gives
\ts $e(k^k) \ge \Phi(k)\. \Phi(k-1)$,
which implies
$$(\sqsm) \qquad \ \ \log \ts e(k^k) \, \ge \, \log \Phi(k) \. +  \. \log \Phi(k-1) \, = \,
\frac{1}{2} \. n \ts \log n \. - \. \frac{3}{2} \. n \. + \. O\bigl(\sqrt{n}\ts\log n\bigr) \..
$$
For the upper bound, observe that $\cq_n=(k^k)$ can be decomposed into chains of
lengths $2k-1,\ldots,3,1$, each involving two adjacent diagonals. We have then:
$$(\oplus) \qquad \aligned
\log \ts e(k^k) \, & \le \, \log \binom{n}{2k-1,\ts 2k-3, \ts \ldots\ts, \ts 3,\ts 1} \, =
\, \log n! \, - \, \log \Psi(k) \\
& \le \,
\frac{1}{2} \. n \ts \log n \. + \left(\frac12 - \log 2\right)n \.
+ \. O\bigl(\sqrt{n}\ts\log n\bigr) \..
\endaligned
$$
In other words, the lower and upper bounds agree in the leading term of the
asymptotics but not in the second term.
Let us compare this with an exact value of $e(k^k)$.  The HLF gives:
$$
e(k^k) \, = \, \frac{n! \. \Phi(k-1)^2}{\Phi(2k-1)}\.,
$$
which implies
$$
\log \ts e(k^k) \, = \, \frac{1}{2} \. n \ts \log n \. + \. \left(\frac12 - 2 \log 2\right)
\. n \. + \. O\bigl(\sqrt{n}\ts\log n\bigr).
$$
Since \ts $\left(\frac12 - 2 \log 2\right) \ts \approx \ts -0.8863$ \ts and
\ts $\left(\frac12 - \log 2\right) \ts \approx \ts -0.1931$, we conclude that the
true constant $-0.8863$ of the second asymptotic term lies roughly halfway between the
lower bound $-1.5$ in~$(\sqsm)$ and the upper bound $-0.1931$ in~$(\oplus)$.
Let us mention also that the lower bound $\ts e(k^k)\ge n!/(k!)^{2k}\ts $
in Theorem~\ref{t:poset-ineq} is much too weak.

\medskip

\subsection{Skew shapes}\label{ss:posets-syt}
Let $\la/\mu$ be a skew shape Young diagram (see Figure~\ref{f:exbounds}).
To simplify the notation, we use $\la/\mu$ to also denote the corresponding
posets of squares increasing downward and to right.
The main object of this paper is the asymptotic analysis of
$$e(\la/\mu) \. = \. f^{\la/\mu} \. = \. |\SYT(\la/\mu)|\ts,
$$
the number of standard Young tableaux
of shape~$\la/\mu$. Note that both are standard notation in different areas;
we use them interchangeably throughout the paper.

The following determinant formula due to Feit~\cite{Feit}
is a standard result in the area, often referred to as the
\emph{Jacobi--Trudi identity} (see e.g.~\cite{Sag,EC2}):
$$(\triangledown) \qquad
f^{\la/\mu} \, =  \, n!\.
\det \left( \frac{1}{(\la_i-\mu_j -i+j)!}\right)_{i,j=1}^{\ell(\la)}.
%\quad \text{where} \ \ \ell=\ell(\la)\ts.
$$
Unfortunately, due to the alternating sign nature of the determinant,
this formula is difficult to use in the asymptotic context.  Here is
the only (quite weak) general bound that easily follows from the
existing literature.

\begin{proposition} \label{p:skew-bound}
For every skew shape $\la/\mu$, we have:
$$
f^{\la/\mu} \, \le \, \frac{|\la|! \ts f^\mu}{|\mu|!\ts f^\la}
\, = \, \frac{\prod_{u\in \la} \ts h(u)}{\prod_{v\in \mu} \ts h(v)}\..
$$
\end{proposition}

\begin{proof}
Recall the standard equalities for the \emph{Littlewood--Richardson {\rm (LR--)} coefficients}
\ts $c^\la_{\mu\ts\nu}$, where \ts $|\la| = |\mu|+|\nu|$.  We have:
$$  % (\divideontimes) \qquad
f^\mu \ts f^\nu \ts \binom{|\mu|+|\nu|}{|\mu|} \, = \, \sum_{\la\vdash |\mu|+|\nu|} \. c^\la_{\mu\ts\nu}\ts f^\la
\quad \text{and} \quad f^{\la/\mu} \, = \, \sum_{\nu \vdash |\la|-|\mu|} \. c^\la_{\mu\ts\nu}\ts f^\nu\..
$$
From here and the Burnside identity~\cite{Sag,EC2}, we have:
$$
f^{\la/\mu} \, = \, \sum_\nu \. c^\la_{\mu\ts\nu}\ts f^\nu \, \le \,
 \sum_\nu \. \binom{|\la|}{|\mu|}\.\frac{f^\mu \. f^\nu}{f^\la} \. f^\nu
 \, = \, \frac{f^\mu}{f^\la} \. \binom{|\la|}{|\mu|} \. \sum_\nu \. (f^\nu)^2
 \, = \, \frac{f^\mu \. |\la|!}{f^\la \. |\mu|!\.|\nu|!} \. |\nu|!\, = \,
 \frac{|\la|! \ts f^\mu}{|\mu|!\ts f^\la} \.,
$$
proving the first inequality.  The equality follows from the HLF.
\end{proof}

\begin{example}\label{e:posets-ex}
Let $\lambda/\mu = (4^232/21)$ be a skew shape of size $n=10$.
The Jacobi--Trudi formula~$(\triangledown)$ gives
$f^{\la/\mu}=e(\lambda/\mu) = 3060$.

For the poset $\lambda/\mu$, the length of the
longest antichain is $m=4$, and the number of elements of each rank are $3,4,3$.
Similarly, the poset can be decomposed into four chains of sizes $3,3,3,1$
(see Figure~\ref{f:exbounds}).  Theorem~\ref{t:poset-gen} gives:
$$
864 \. = \. 3! \ts 4!\ts 3! \, \leq \, e(\lambda/\mu)
\, \leq \, \frac{10!}{3!\ts 3!\ts 3! \ts 1!} \. = \.  16800\ts.
$$
The sizes $\br(x)$ are given in Figure~\ref{f:exbounds}. Then
Theorem~\ref{t:poset-ineq} gives a slightly weaker bound:
$$
e(\lambda/\mu) \, \geq \,
\frac{10!}{2^23^25^26} \, = \, 672\ts.
$$
Finally, Proposition~\ref{p:skew-bound} and the HLF gives the following
very weak upper bound:
$$
e(\lambda/\mu) \, \leq \, \frac{13!\. e(21)}{3!\. e(4^232)} \, = \,\frac{13!\cdot 2}{3!\cdot 8580} \, = \, 241920\ts.
$$
\end{example}
\begin{figure}
\includegraphics{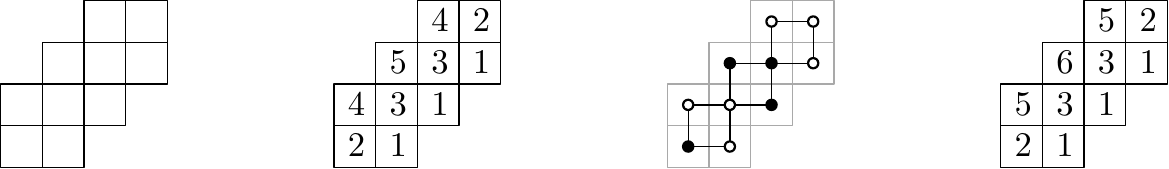}
\caption{The shape $\lambda/\mu=(4^232/21)$, its hook-lengths,
the Hasse diagram of poset~$\cP$, and the sizes of the upper
ideals spanned by each element of the poset.}
\label{f:exbounds}
\end{figure}

\begin{remark} \label{r:posets-trunkated}
Calculations similar to the square shape
can be done for various other geometric shapes with known
nice product formulas.  Beside the usual Young diagrams these include \emph{shifted
diagrams} and various ad hoc shapes as in Figure~\ref{f:shifted} (see~\cite{AR,KS,Pan}
and the last example in~\cite{MPP3}).  In all these cases,
$$(\heartsuit) \qquad \log e(\cP) \, = \,  \frac{1}{2} \ts n \log n \. + \. O(n)\..
$$
While the bounds in Theorem~\ref{t:poset-gen} again give the leading term correctly,
they are all off in the second asymptotic term.   This observation is the key
starting point for this work.  Roughly speaking, in many cases,
the inequalities~$(\ast)$ in the introduction make the gap between the second
asymptotic term smaller.
\end{remark}

\begin{figure}[hbt]
\includegraphics[width=12.8cm]{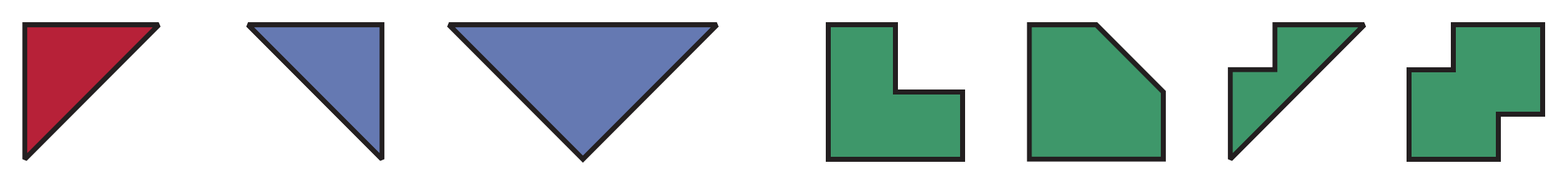}
\caption{Staircase shape, two shifted shapes and three ad hoc shapes
with product formulas.}
\label{f:shifted}
\end{figure}

\bigskip

\section{Hooks formulas for skew shapes}\label{s:def}

\subsection{Definition of excited diagrams}
Let $\lambda/\mu$ be a skew partition and $D$ be a subset of the Young
diagram of $\lambda$. A cell $u=(i,j) \in D$ is called {\em active} if
  $(i+1,j)$, $(i,j+1)$ and $(i+1,j+1)$ are all in
$\lambda\setminus D$.  Let $u$ be an
active cell of $D$, define $\alpha_u(D)$ to be the set obtained by
replacing $(i,j)$ in $D$ by $(i+1,j+1)$. We call this replacement an {\em excited move}.
An {\em excited diagram} of
$\lambda/\mu$ is a subset of squares in $\lambda$ obtained from the Young
diagram of $\mu$ after a sequence of excited moves on active cells. Let
$\mathcal{E}(\lambda/\mu)$ be the set of
excited diagrams of~$\lambda/\mu$, and let $\ed(\la/\mu)=|\ED(\lambda/\mu)|$.

The following explicit characterization in~\cite{MPP1} is also very helpful.
Let $D$ be a subset of squares in $\la$ with the same number of squares
in each diagonal as~$\mu$.  Define an order relation on squares $(i,j) \preccurlyeq (i',j')$
if and only if $i\le i'$ and $j\le j'$.  Then~$D$ is an excited diagram if and only
if the relation~$\preccurlyeq$ on squares of~$\mu$ hold for~$D$.

We conclude with an explicit formula for $\ed(\la/\mu)$. For the diagonal~$\De$ that
passes through the cell $(i,\mu_i)$, denote by $\vt_i$ the row in
which~$\De$ intersects the boundary of~$\lambda$.

\begin{theorem}[\cite{MPP1}]\label{t:ed-det}
Let $\la/\mu$ be skew partition and let $\ell=\ell(\mu)$. In the notation above, we have:
$$
\ed(\la/\mu) \, = \, \det\left[\binom{\vt_i+\mu_i-i+j-1}{\vt_i-1}\right]_{i,j=1}^{\ell}
$$
\end{theorem}

This formula follows from a characterization of the excited diagrams
as certain \emph{flagged tableaux} of shape $\mu$ with entries in row
$i$ at most $\vt_i$, a different
border strip decomposition, and the Lindstr\"om--Gessel--Viennot lemma.  We refer
to~\cite{MPP1,MPP2} for details and the references.

\medskip

\subsection{NHLF and its implications}
The following recent result is the crucial advance which led to
our study (cf.~\cite{MPP1,MPP2}).

\begin{theorem}[Naruse \cite{Nar}] \label{thm:IN}
Let $\lambda,\mu$ be partitions, such that $\mu \ssu \la$.  We have:
\begin{equation} \label{eq:Naruse} \tag{NHLF}
e(\lambda/\mu) \,  = \, |\la/\mu|! \, \sum_{D \in \ED(\lambda/\mu)}\,\.\.
 \prod_{u \in \lambda\setminus D} \frac{1}{h(u)}\ts\..
\end{equation}
\end{theorem}

We can now present a corollary of the NHLF, which is the main technical tool
of this paper.  For a general bound of this type, it is quite
powerful in applications (see below).  It is also surprisingly
easy to prove.

\begin{theorem}  \label{t:main}
For every skew shape $\la/\mu$, $|\la/\mu|=n$, we have:
$$(\ast) \qquad \ \
F(\la/\mu) \, \leq \, e(\lambda/\mu) \, \leq \, \ed(\lambda/\mu) \. F(\la/\mu)\,,
$$
where
$$F(\la/\mu) \, = \, n! \. \prod_{u\in \lambda/\mu} \. \frac{1}{h(u)}
$$
is defined as in the introduction.
\end{theorem}

Note that when restricted to skew shapes,
the lower bound in Theorem~\ref{t:poset-ineq} is clearly weaker than
the lower bound~$(\ast)$, and coincides with it exactly for the
\emph{ribbon hooks} (shapes with at most one
square in every diagonal).

\begin{proof}
The lower bound follows from the NHLF, since $\mu$ is an
excited diagram in $\ED(\lambda/\mu)$.  \ts
For the upper bound, note that under the excited
move the product $\prod_{u\in \lambda\setminus D} \ts h(u)$ increases.
Thus this product is minimal for $D=\mu$, and the upper bound follows
from the NHLF.
\end{proof}

\begin{remark}
Note also that both bounds in Theorem~\ref{t:poset-gen} are symmetric
with respect to taking a dual poset.  On the other hand, one can apply
Theorem~\ref{t:main} to either $(\la/\mu)$ or the dual shape
$(\la/\mu)^\ast$ obtained by a $180$ degrees rotation.
\end{remark}

\begin{example}
As in Example~\ref{e:posets-ex}, let $\lambda/\mu = (4^232/21)$ and
$e(\lambda/\mu) = 3060$.  The hook-lengths in this case are given
in Figure~\ref{f:exbounds}.  We have
$$
F(\lambda/\mu) \, = \, \frac{10!}{5\ts 4^2\ts 3^2\ts 2^2}  \, = \.  1260\ts.
$$
In this case we have $\vt_1=2$, $\vt_2=3$, and by Theorem~\ref{t:ed-det}
$$\ed(\lambda/\mu) \, = \,
\det \begin{pmatrix} 3 & 4 \\
1 & 3 \end{pmatrix}
\, = \, 5\ts.
$$
In this case, Theorem~\ref{t:main} gives:
$$
1260 \, \leq \,  e(\lambda/\mu) \, \leq \, 1260\cdot 5 \,=\, 6300\ts.
$$
Both bounds are better than the bounds in Example~\ref{e:posets-ex}.
\end{example}

\medskip

\subsection{Comparison of bounds}\label{ss:posets-compare}
To continue the theme of this section, we make many comparisons between the bounds
on $e(\la/\mu)$ throughout the paper, both asymptotically and in special cases.
Here is the cleanest comparison, albeit under certain restrictions.

Fix a skew shape $\ups=(\la/\mu)$ with $\ell=\lc(\la)$.  Let $A_1,\ldots,A_\ell$
be an antichain decomposition with elements in $A_k$
lying in the antidiagonal $\ts A_k=\{(i,j) \in \ups \ts\mid\ts i+j=k+s-1\}$, where
$\ts s = \min\{\ts i+j \ts\mid\ts (i,j) \in \ups\}$.  Denote $\ts r_k=|A_k|$.

\begin{theorem}\label{t:compare}
In the notation above, suppose $r_1\le r_2\le \ldots \le r_k$.  Then we have:
$$r_1! \ts \cdots \ts r_k! \. \leq \. F(\lambda/\mu)\ts.$$
\end{theorem}

For example, the theorem applies to the thick ribbon shapes
(see~$\S$\ref{s:ribbons}).

\begin{proof}
Consider the hooks of the squares on a given antidiagonal~$A_i$. Let the number of
squares of rank $\geq i$ (i.e.\ on the antidiagonal or below/right of it) be $N_i$.
Every such square belongs to at most 2 hooks $h_u$ and $h_v$ with $u,v \in A_i$,
and that happens only when it is contained in the ``subtriangle'' with antidiagonal
$A_i$ (everything left after erasing rows and columns which have no boxes in $A_i$).
There are at most $\binom{r_i}{2}$ boxes in such a rectangle, and cannot be more
than all boxes below diagonal~$i$, which is $N_i-r_i$. In fact, since
$\ts r_{i+1} \geq r_i$, there is at least one box below the diagonal
which is not counted twice, so the bound is $N_i-r_i-1$.
Hence by counting the squares covered by a hook, we have:
$$
\sum_{u \in A_i}\. h_u \, \leq \, N_i \. + \. \min\left\{ N_i-r_i-1, \binom{r_i}{2}\right\}\ts.
$$
Noting that for the last diagonal we have $h_u=1$ and $N_k=r_k$, we obtain the following
$$(\maltese) \qquad \quad
\aligned
\prod_{i=1}^k r_i! \. \prod_i \. \prod_{u \in A_i} \.h_u \, & \leq  \, \prod_{i=1}^k r_i! \. \prod_{i=1}^{k-1} \ts
\left( \frac{ \sum_{u \in A_i} h_u }{r_i} \right)^{r_i} \\
& \leq r_k! \prod_{i=1}^{k-1} \frac{ r_i! \left( N_i +\min\bigl\{N_i-r_i-1, \binom{r_i}{2} \bigr\} \right)^{r_i}}{ r_i^{r_i}}\..
% \leq \prod_i r_i! \binom{ N_i}{r_i}  = N_1! = N!
\endaligned
$$
For the first inequality is the AM--GM inequality for the product of hooks on a given diagonal,
and for the second -- the estimate for their sum.

Next we will need the following inequality for binomial coefficients.
The proof is straightforward.

\begin{lemma}\label{l:compare-binom}
Let  $t \geq r \geq 3$.  Then
%\begin{equation}%\label{ineq:binom}
$$
\binom{t+r}{r} \.\geq \. \left(\frac{2t+r-1}{r}\right)^r.
$$
%\end{equation}
\end{lemma}

\smallskip

We use the lemma to estimate the RHS of inequality~(\maltese).
For $r_i \geq 3$, take $t=N_i - r_i \geq r_i$ for $i<k$.  We have:
$$
\frac{ (2N_i -1-r_i)^{r_i}}{r_i^{r_i} } \. \leq \.\binom{N_i}{r_i}\ts.
$$
For $r_i=2$, we have $\ts \min(N_i-2, 1) \leq 1\ts$ and the inequality becomes
$$\frac{( N_i + \min\{N_i-2,1\} )^2}{2^2} \, \leq\, \frac{(N_i+1)^2}{4} \leq \binom{N_i}{2},
$$
which holds trivially whenever $N_i \geq 5$. For $N_i \leq 4$ and $r_i=2$,
under the assumptions, we must have a partition of size at most~5;
such cases are checked by direct computation.
Finally, for $r_i=1$, we have $\ts\min(N_i -1,0) =0$, and we have the equality.
Putting all these together in the RHS of~$(\maltese)$, we have
$$ \prod_{i=1}^k \ts r_i! \. \prod_u \ts h_u \,
\leq \, \prod_{i=1}^k \ts r_i! \. \prod_{i=1}^k \ts \binom{N_i}{r_i} \, = \, N_1!\.,
$$
since $N_i = r_i + N_{i+1}$ and $N_1=N.$ Dividing both sides
by the product of hooks gives the desired inequality.
\end{proof}

\begin{remark}
The antichain decomposition in the theorem can be generalized
from rank antichains to all ordered antichain decompositions
$(A_1,\ldots,A_\ell)$, such that $x\in A_i$, $y\in A_j$, $i<j$,
then $x\prec y$.  The proof extends verbatim; we omit the details. \ts
Note also that the theorem cannot be extended to general skew shapes.
The examples include $(2^2/1)$, $(3^3/321)$ and $(4^4/2^2)$.
% The latter example generalizes to all inverted thick hooks
% in~$\S$\ref{ss:examples-box}.
\end{remark}
%\vfill\eject

\bigskip

\section{Bounds on the number of excited diagrams} \label{s:lemmas}

\subsection{Non-intersecting paths}
We recall that the excited diagrams of~$\lambda/\mu$ are in bijection
with families of certain non-intersecting grid paths $\ga_1,\ldots,\ga_k$ with a fixed set of
start and end points, which depend only on~$\la/\mu$.  This was proved in~\cite{MPP1,MPP2},
and based on the earlier works by Kreiman~\cite{Kre}, Lascoux and Sch\"utzenberger~\cite{LaS},
and Wachs~\cite{Wac}, on flagged tableaux.

\smallskip

Formally, given a connected skew shape $\lambda/\mu$, there is unique family of
border-strips (i.e. non-intersecting paths)
 $\ga^*_1,\ldots,\ga^*_k$ in $\lambda$ with support
$\lambda/\mu$, where each border strip $\ga^*_i$ begins at the southern box
$(a_i,b_i)$ of a column and ends at the eastern box $(c_i,d_i)$ of
a row
%, and the border strips are ordered by
%$\ga^*_i \preceq \ga^*_j$ if $(c_i,a_i) \leq (c_j,a_j)$ component-wise
\cite[Lemma 5.3]{Kre}. Moreover, all non-intersecting paths
$(\ga_1,\ldots,\ga_k)$ contained in $\lambda$ with
$\ga_i:(a_i,b_i)\to (c_i,d_i)$ are in correspondence with excited
diagrams of the shape $\lambda/\mu$ \cite[$\S$5.5]{Kre}.

\begin{proposition}[Kreiman~\cite{Kre}, see also \cite{MPP2}]
The non-intersecting paths $(\gamma_1,\ldots,\gamma_k)$ in $\lambda$
where $\gamma_i: (a_i,b_i) \to (c_i,d_i)$ are uniquely determined by
their support and moreover these supports are in bijection with complements of
excited diagrams of~$\lambda/\mu$.
\end{proposition}

\begin{corollary}  In the notation above,
\[
\ed(\lambda/\mu) \, = \, \#\{\text{non-intersecting paths }
(\ga_1,\ldots,\ga_k) \mid \ga_i \subseteq \lambda,\, \ga_i:(a_i,b_i)\to (c_i,d_i)\}\ts.
\]
\end{corollary}

See Figure~\ref{f:non-int} for an example of the proposition and the corollary.

\begin{figure}[hbt]
\includegraphics{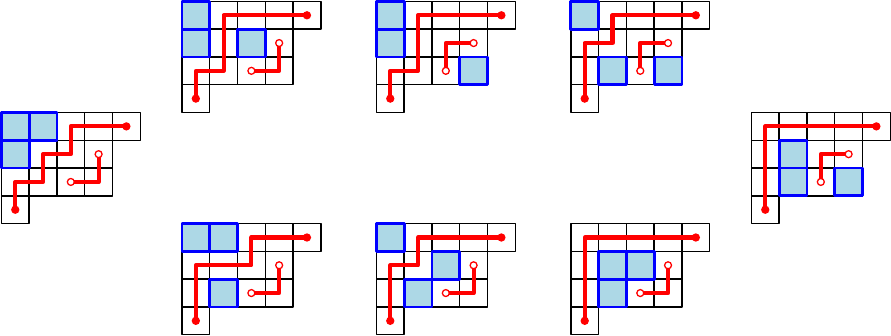}
\caption{The eight non-intersecting paths $(\ga_1,\ga_2)$ whose
  support are the complement of the excited diagrams of the shape
  $(54^21/21)$.}
\label{f:non-int}
\end{figure}

% \begin{remark}
% By the Lindstr\"om-Gessel-Viennot lemma this implies that
% $\ed(\lambda/\mu)$ is given by determinant,
% \begin{equation} \label{eq:LGV}
% \ed(\lambda/\mu) \, =\, \det{\big [} \.\#\text{ paths } \, \ga: (a_i,b_i) \to (c_j,d_j), \, \gamma \subseteq
% \lambda\. {\big ]}_{i,j=1}^k.
% \end{equation}
% Note that this is a different determinantal formula for $\ed(\lambda/\mu)$
% than the one given in \cite[$\S$3]{MPP1}. We refer to~\cite{MPP2} for
% further details, examples and references.
% \end{remark}

\medskip

\subsection{Bounds on $\ed(\la/\mu)$}
We prove the following two general bounds. They are elementary
but surprisingly powerful in applications.

\begin{lemma}\label{l:exp}
Let $\la/\mu$ be a skew shape, $n=|\la/\mu|$.  Then $\ed(\la/\mu)\le 2^n$.
\end{lemma}

\begin{proof}
Each path of fixed size is determined by its vertical and horizontal steps between its fixed endpoints.  We conclude:
$$
\ed(\la/\mu) \. \le \. \prod_{i=1}^k \. 2^{|\ga_i|-1}  \. = \. 2^{n-k} \. \le \. 2^{n}\.,
$$
where $k$ is the number of paths $\ga_i$, as above.
\end{proof}

\smallskip

The \emph{Durfee square} in Young diagram~$\la$ is the maximal square
which fits~$\la$.  Let $d(\la)$ be the size of the Durfee square.

\begin{lemma}\label{l:poly}
Let $\la/\mu$ be a skew shape and let $d=d(\la)$.  Then $\ed(\la/\mu)\le n^{2d^2}$.
\end{lemma}

\begin{proof}
Since $\la/\mu$ is connected, the path traced by the border of $\la$ is a ribbon of length at most $n$ and height at least $d$, so  we must have that $\la_1, \ell(\la) \leq n-d$. Then $\la \subset \nu$, where $\nu=\bigl( (n-d)^d, d^{n-2d}\bigr)$. This implies $\ed(\la/\mu) \leq \ed(\nu/\mu)$.

As before, we estimate $\ed(\nu/\mu)$ by the number of non-intersecting lattice paths inside that shape. Consider the lattice paths above the diagonal $i=j$. Since the shape $\nu/\mu$ within the top $d$ rows is the complement of a straight shape, the lattice paths there all start at the bottom row and follow the shape $\mu$, ending at the rightmost column. So there are at most $d$ paths above the diagonal and their endpoints are no longer fixed (see Figure~\ref{f:paths}).
Clearly, each such path has length at most~$n$.  Those above diagonal,
have at most $(d-1)$ vertical steps and those below diagonal have at most
$(d-1)$ horizontal steps.  Therefore,
$$
\ed(\la/\mu) \. \le \. \ed(\nu/\mu) \. \le \. \binom{n-1}{d-1}^{2d} \. \le \. n^{2d^2}\.,
$$
as desired.
\end{proof}

\begin{figure}
\includegraphics[width=12.4cm]{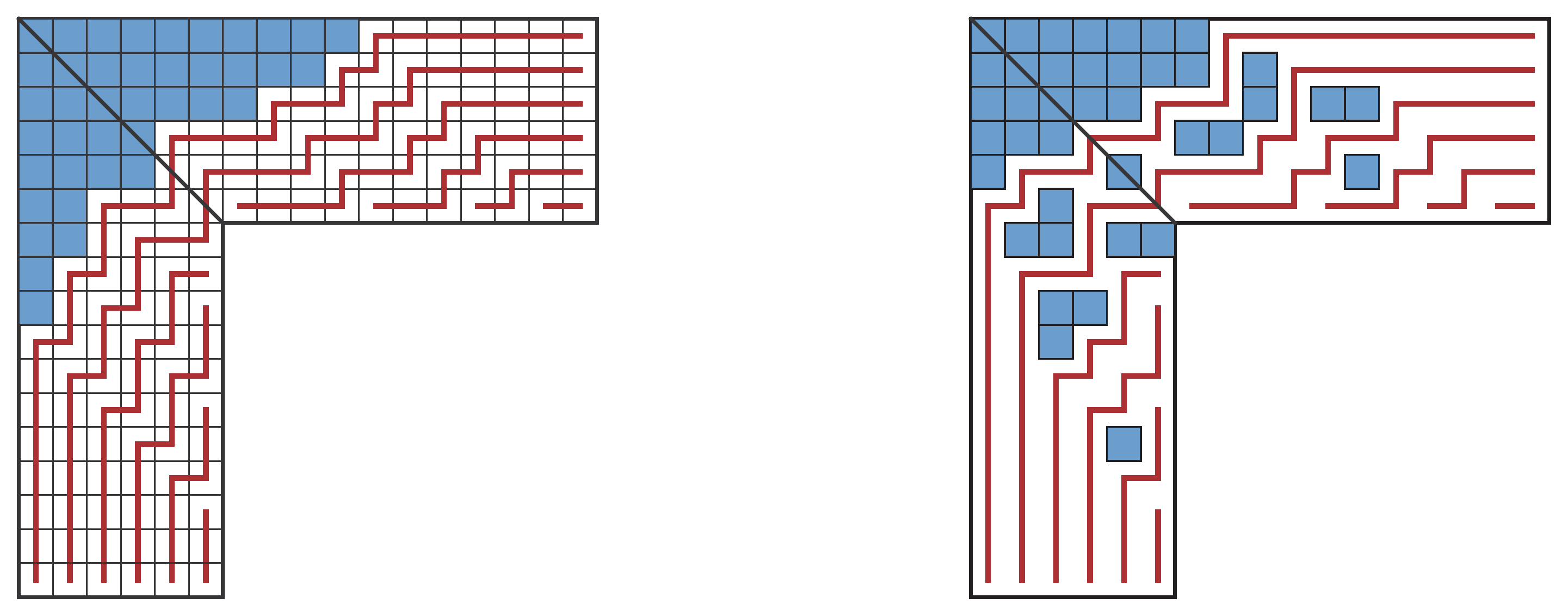}
\caption{Sets of paths in $\nu/\mu$ corresponding to excited diagrams. }
\label{f:paths}
\end{figure}

\bigskip

\section{The Thoma--Vershik--Kerov limit shape} \label{s:vk}

Denote by $a(\la)=(a_1,a_2,\ldots)$, $b(\la)=(b_1,b_2,\ldots)$ the Frobenius
coordinates of~$\la$.  Fix an integer $k\ge 1$. Let $\bal=(\al_1,\ldots,\al_k)$,
$\bbe = (\be_1,\ldots, \be_k)$ be fixed sequences in $\rr_+^k$.
We say that a sequence of partitions $\{\la^{(n)}\}$
has a \emph{Thoma--Vershik--Kerov} (TVK) \emph{limit} \ts $(\ov{\al},\ov{\be})$, write
$\la^{(n)} \. \to \. (\ov{\al},\ov{\be})$, if $a_i/n \to \al_i$ and
$b_i/n\to \be_i$ as $n \to \infty$, for all $1\le i\le k$, and
$d\bigl(\la^{(n)}\bigr)\le k$ for all $n\ge 1$.

\begin{theorem}
Let $\{\ups_n = \la^{(n)}/\mu^{(n)}\}$ be a sequence of skew shapes with a TVK limit.
Formally, in the notation above, suppose $\la^{(n)} \. \to \. (\ov{\al},\ov{\be})$,
where $\al_1,\be_1>0$, and $\mu^{(n)} \. \to \. (\ov{\pi},\ov{\tau})$ for
some $\ov{\al},\ov{\be},\ov{\pi},\ov{\tau}\in \rr_+^k$.  Then
$$
\log e(\ups_n) \, = \, c\ts n + o(n) \quad  \text{as} \quad n \to \infty\.,
$$
where
$$
c=c(\ov{\al},\ov{\be},\ov{\pi},\ov{\tau}) \, = \, \ga\ts \log\ga \. - \. \sum_{i=1}^k \. (\al_i - \pi_i)\ts \log(\al_i - \pi_i) \. - \.
\. \sum_{i=1}^k \. (\be_i - \tau_i)\ts \log (\be_i - \tau_i)\.
$$
and
$$\ga \. = \. \sum_{i=1}^k \, (\al_i+\be_i-\pi_i-\tau_i)\..$$
\end{theorem}

\begin{proof}
First, in the notation of TVK limit, we have $d\bigl(\la^{(n)}\bigr) \le k$, so
$\ed(\ups_n)< n^{2k^2}$.  Thus, by Theorem~\ref{t:main},
it suffices to compute the lead term in
the asymptotics of $F(\la^{(n)}/\mu^{(n)})$. Observe that
$|\ups_n| \. \sim \.  \ga \ts n$.
Next, observe that in the $i$-th row, up to an additive constant $\le d(\la) = k$, all the hooks are
$1, \ldots, (\al_i - \pi_i + o(1))n$,
and in the $i$-th column the product of hooks is  $1, \ldots, (\be_i - \tau_i + o(1))n$. Therefore,
by the Stirling formula, we have:
$$
\log F(\ups_n) \,  \sim \, \log (\ga n)! \. - \. \sum_i \log \bigl[(\al_i - \pi_i)n\bigr]!
\. - \. \sum_i \log \bigl[(\be_i - \tau_i)n\bigr]!
$$
$$\sim \, n\left(\ga\log\ga \. - \. \sum_i \. (\al_i - \pi_i)\log(\al_i - \pi_i) \. - \.
\. \sum_i \. (\be_i - \tau_i) \log (\be_i - \tau_i)
\right),
$$
as $n \to \infty$.  We omit the easy details.
\end{proof}

\begin{remark}
When $\mu=\emp$, the exponential nature of $f^\la$ was recently studied in~\cite{GM};
one can view our results as a asymptotic version in the skew shape case.
For a fixed partition~$\mu$, Okounkov and Olshanski~\cite{OO} give an explicit formula for the
ratio $f^{\la/\mu}/f^\la$ which can be computed explicitly in this case
(see also~\cite{CGS,Stanley_skewSYT}).
Note also that the notion of TVK limit is rather weak, as it is oblivious
to adding $s=o(n)$ to the rows (columns), which can affect $f^{\la/\mu}$
by a factor of $2^{\Theta(s)}$.  Thus the error term in the theorem cannot
be sharpened.

Note finally that $d\bigl(\la^{(n)}\bigr) = O(1)$ property is essential in the theorem,
as otherwise $\ed(\la/\mu)$ can be superexponential (see Section~\ref{s:slim}).
\end{remark}

\bigskip

\section{The stable shape} \label{s:stable}

\subsection{The usual stable shape} % \ts
Let $\om: \rr_+\to \rr_+$ be a non-increasing continuous function.
Consider a sequence of partitions $\{\la^{(n)}\}$, such that
the \emph{rescaled diagrams} \ts $\frac{1}{\sqrt{n}}\bigl[\la^{(n)}\bigr]$ \ts converge uniformly to~$\om$,
where $[\la]$ denotes the curve giving the boundary of Young diagram~$\la$, which we contract by
a factor $1/\sqrt{n}$ in both directions.  In this case we say that a sequence of partitions $\{\la^{(n)}\}$
has a \emph{stable shape $\om$}, and write $\la^{(n)} \to \om$.

Suppose we are given two stable shapes $\om,\pi: \rr_+\to \rr_+$,
such that $\pi(x) \le \om(x)$ for all $x\ge 0$.  To simplify the notation,
denote by $\cC=\cC(\om/\pi) \ssu \rr_+^2$ the region between the curves.
One can view $\cC$ as the stable shape of skew diagrams, and denote
by $\area(\om/\pi)$ the area of~$\cC$.

\begin{theorem}\label{t:stable}
Let $\om,\pi: [0,a] \to [0,b]$ be continuous non-increasing functions,
and suppose that $\area(\om/\pi)=1$. Let $\{\ups_n=\la^{(n)}/\mu^{(n)}\}$
be a sequence of skew shapes with the stable shape $\om/\pi$,
i.e. $\la^{(n)} \to \om$, $\mu^{(n)} \to \pi$.  Then
$$\log e(\ups_n) \, \sim \,
\frac{1}{2}\. n\ts \log n \,  \quad  \text{as} \quad n \to \infty\ts.
$$
\end{theorem}

We prove the theorem below.  Let us first state a stronger result
for a more restrictive notion of stable shape limit, which shows how
the stable shape $\om/\pi$ appears in the second term of the asymptotic
expansion.

\begin{remark}
Note that as stated, the theorem applies to nice geometric shapes
with $\sqrt{n}$ scaling and the Plancherel shape~$\om$, see~\cite{VK}
(see also~\cite{Rom,Stanley_ICM}), but not the Erd\H{o}s--Szekeres
limit shape (see e.g.~\cite{DVZ}), since random partitions have
$\Theta(\sqrt{n} \ts\log n)$ parts.  The proof, however,
can be adapted to work in this case (see also~$\S$\ref{ss:fin-shifted}).
\end{remark}

% \begin{figure}[hbt]
% \epsfig{file=stable.pdf, width=11.0cm}
% \caption{Two stable shapes which fit a square, a stable shape
% of thick ribbons and inverted thick hooks.}
% \label{f:stable}
% \end{figure}

\medskip

\subsection{Strongly stable shape}
Let $\om: \rr_+\to \rr_+$ be a non-increasing continuous function.
Suppose sequence of partitions $\{\la^{(n)}\}$ satisfies the
following property
$$
(\sqrt{n}-L)\ts\om \. < \.  \bigl[\la^{(n)}\bigr]\. < \.  (\sqrt{n}+L)\ts\om\ts,
\ \quad \text{for some} \ \ L>0\ts,
$$
where we write $[\la]$ to denote a function giving the boundary of
Young diagram~$\la$.
In this case we say that a sequence of partitions $\{\la^{(n)}\}$
has a \emph{strongly stable shape $\om$}, and write $\la^{(n)} \mapsto \om$.

In the notation above, define the \emph{hook function} $\hba: \cC\to \rr_+$
to be the scaled function of the hooks:  \ts
$\hba(x,y):= h\bigl(\lfloor x\sqrt{n}\rfloor,\lfloor y\sqrt{n}\rfloor\bigr)$.

\begin{theorem}\label{t:stable-strong}
Let $\om,\pi: [0,a] \to [0,b]$ be continuous non-increasing functions,
and suppose that $\area(\om/\pi)=1$. Let $\{\ups_n=\la^{(n)}/\mu^{(n)}\}$
be a sequence of skew shapes with the strongly stable shape $\om/\pi$, i.e.\
$\la^{(n)} \mapsto \om$ and $\mu^{(n)} \mapsto \pi$.
Then
% $$\log \left( f^{\la^{(n)}/\mu^{(n)}}\right) \, \sim \,
% \frac{1}{2}\ts n\ts \log n \,  \quad  \text{as} \quad n \to \infty\ts.
% $$
% Moreover,
$$-\bigl(1 \ts + \ts c(\om/\pi)\bigr)\ts n \. + \. o(n)\,\le \,
\log e(\ups_n) \. - \.
\frac{1}{2}\ts n\ts \log n \, \le \, - \ts \bigl(1 \ts + \ts c(\om/\pi)\bigr)\ts n
\. + \. \log \ed(\ups_n) \. + \. o(n)\ts,
$$
as \ts $n \to \infty$, where
$$
c(\om/\pi) \. = \. \iint_\cC \. \log \hba(x,y) \, dx \. dy\ts.
$$
\end{theorem}

Note that by Lemma~\ref{l:exp}, we always have $\log \ed(\ups_n) \le (\log 2)\ts n$.

\begin{proof}
Theorem~\ref{t:main} gives:
$$
\log \ts F(\ups_n) \. \le \.
\log e(\ups_n) \. \le \. \log \ts F(\ups_n) \.
+ \. \log \ed(\ups_n)\ts.
$$
Now, observe that \ts $|\ups_n| =  n + O(\sqrt{n})$ \ts as $n \to \infty$.
Using the Stirling formula, the definition and compactness of the stable shape
$\cC\ssu [a\times b]$,\footnote{In~\cite{FeS}, such shapes are called \emph{balanced}.}  we have:
$$
\log \ts F(\ups_n) \. =  \. |\ups_n|! \ - \. \sum_{u \in \ups_n} \. \log \ts h(u) \.
= \.  n \ts \log n \. - \. n \.  + \. O(\sqrt{n} \ts \log n) \.  - \qquad
$$
$$\qquad -\.  n\cdot \iint_\cC \. \log \bigl(\sqrt{n} \. \hba(x,y)\bigr) \, dx \. dy \.  + \. o(n) \.
= \. \frac{1}{2} \. n \ts \log n \. - \. n \. - \. c(\om/\pi)\ts n \.  + \. o(n)\ts,
$$
where the $o(n)$ error term comes from approximation of the sum with the scaled integral.  This
implies both parts of the theorem.
\end{proof}

\begin{proof}[Proof of Theorem~\ref{t:stable}]
The proof is similar, but the error terms are different.
First, by uniform convergence, we have
\ts $|\ups_n| =  n + o(n)$, which only implies
$$
|\ups_n|! \,=\, \frac12 \. n\log n + o(n\log n) \, \quad  \text{as} \quad n \to \infty\ts.
$$
The same error term $o(n\log n)$ also appears from the scaled integral bound.  The details
are straightforward.
\end{proof}

\bigskip

\section{The subpolynomial depth shape} \label{s:sub-poly}

%\subsection{The result}
Let $g(n): \nn \to \nn$ be an integer function which satisfies $1\le g(n) \le n$, $g(n) \to \infty$ and
$\log g(n) = o(\log n)$ as $n \to \infty$.  The last condition is equivalent
to $g(n) = n^{o(1)}$, hence the name. Function $g(n)$ is said to have \emph{subpolynomial growth}
(cf.~\cite{GP}).  Examples include $e^{\sqrt{\log n}}$, $(\log n)^\pi$, $n^{1/\log \log n}$, etc.

For a skew shape $\la/\mu$, define $\ts\width(\la/\mu):=\min\{\la_1,\la_1'\}$,
and $\ts\depth(\la/\mu):=\max_{u\in \la/\mu} \ts h(u)$.  For example, for thick ribbons
$\ups_k=\de_{3k}/\de_{2k}$, we have $\width(\ups_k)=3k-1$ and $\depth(\ups_k) = 2k-1$.

We say that a sequence of skew partitions $\{\nu_n=\la^{(n)}/\mu^{(n)}\}$
has \emph{subpolynomial depth shape} if
$$(\circ)\qquad \
\width(\nu_n)\,= \, \Theta\left(\frac{n}{g(n)}\right) \quad \text{and} \quad
\depth(\nu_n)\,= \, \Theta\bigl(g(n)\bigr)\ts,
$$
where $g(n)$ is a subpolynomial growth function.

\begin{theorem}\label{t:sub-poly}
Let $\ts \{\nu_n=\la^{(n)}/\mu^{(n)}\}\ts $ be a sequence of skew partitions with a
subpolynomial depth shape associated with the function~$g(n)$.
Then
$$
\log e(\nu_n) \, = \, n \log n \. -\. \Theta\bigl(n \ts \log g(n)\bigr) \quad
\text{as} \quad n \to \infty\ts.
$$
\end{theorem}

\begin{proof}  First,  by
Lemma~\ref{l:exp}, we have $\ed(\nu_n) \le 2^n$. Thus, by Theorem~\ref{t:main},
we have:
$$\log e(\nu_n) \, = \, \log F(\nu_n) +O(n)
\, = \, n\log n \ts + \ts O(n) \. - \sum_{u\in \nu_n} \. \log h(u)\ts.
$$
First, observe:
$$
\sum_{u\in \nu_n} \. \log h(u) \, \le \, n \ts \log \bigl[\depth(\nu_n)\bigr]\, = \,
O\bigl(n \ts \log g(n)\bigr).
$$
In the other direction, suppose for simplicity that $\ts \width(\nu_n)=\la_1$.
By~$(\circ)$ and total area count, we have: \ts
of the \ts $\width(\nu_n)$ \ts columns, at least $\al$ proportion of them
have length $\ts >\ts \be \ts \depth(\nu_n)$, for some $\al,\be>0$.
Here the constants $\al,\be$ depend only on the constants implied by the
$\Theta(\cdot)$ notation and independent on~$n$.  We conclude:
$$\aligned
\sum_{u\in \nu_n} \. \log h(u) \, & \ge \, \al\ts\width(\nu_n)\cdot \log \left[\be \ts \depth(\nu_n)\right]! \\
& \ge \, \al \ts \Theta\left(\frac{n}{g(n)}\right) \cdot \be\ts\Theta\bigl(g(n)\bigr) \ts
\log \bigl[\be \ts \Theta\bigl(g(n)\bigr)\bigr]
\, =\, \Omega\bigl(n \ts \log g(n)\bigr),
\endaligned
$$
as desired.
\end{proof}

\begin{remark}
The subpolynomial depth shape introduced in this section
is a generalization of certain thick ribbon shaped (see below).
It is of interest due largely to the type of asymptotics for $e(\la/\mu)$.
Here the leading term of $\log e(\la/\mu)$ is $n\log n$, which implies that
the lower order terms are of interest.  Furthermore, we have \ts
$\log\ed(\la/\mu) = o\bigl(\log n! - \log F(\la/\mu)\bigr)$ in this case,
which implies that NHLF is tight for the first two terms in the asymptotics.
\end{remark}

\bigskip

\section{Thick ribbons}\label{s:ribbons}

\subsection{Proof of Theorem~\ref{t:ribbons}}  \label{ss:ribbons-NHLF}
As in the introduction, let $\delta_k =(k-1,k-2,\ldots,2,1)$ and
$\ups_k = \de_{2k}/\de_{k}$ (see~\cite{MPP2} and  for the sequence
$\{e(\ups_k)\}$ see \cite[\href{https://oeis.org/A278289}{A278289}]{OEIS}).
Note that thick ribbons $\ups_k$ have strongly stable trapezoid shape.
% (see Figure~\ref{f:stable}).
We have $n=|\ups_k| = 3k(k+1)/2$.
For simplicity, we assume that $k$ is even (for odd~$k$ the formulas are
slightly different). In the notation of $\S$\ref{ss:posets-not}, we have:
$$
F(\ups_k) \, = \, \frac{n!}{(2k-1)!!^{\ts k}\.\. \La(k-1)}\..
$$
This gives the lower bound in Theorem~\ref{t:ribbons}~:
$$
\aligned
\log e(\ups_k) \, & \ge \,
\log F\bigl(\ups_k\bigr) \, = \, \frac{1}{2}\ts n \log n \. + \left(\frac16 \ts - \ts
\frac{3\log 2}{2} \ts + \ts \frac{\log 3}{2} \right)  n + \. o(n) \\
 & \ge \, \frac{1}{2}\ts n \log n \. - \. 0.3237 \ts n \. + \. o(n) \ts.
\endaligned
$$

For the upper bound, we obtain a closed formula for the number of excited diagrams.

\begin{lemma} \label{l:ribbons-excited}
Let $\ups_k=(\de_{2k}/\de_k)$ be a thick ribbon, and let~$k$ be even.
We have:
$$
\ed(\ups_k) \, = \, \prod_{1\le i<j\le k}\. \frac{k+i+j-1}{i+j-1} \. \,
= \, \left[\frac{\HH(3k-1)\. \HH(k-1)^3\. (2k-1)!!\. (k-1)!!}{\HH(2k-1)^3 \ts (3k-1)!!}\right]^{1/2}.
$$
\end{lemma}

\begin{proof}
The product in the RHS is the \emph{Proctor's formula}
for the  number of (reverse) plane partitions of shape
$\delta_k$ with entries $\leq k/2$~\cite[Cases CG]{Pro}.
Let us show that the number also counts excited diagrams
of the shape $\delta_{2k}/\delta_k$. Indeed, by the proof
of Theorem~\ref{t:ed-det} in~\cite{MPP1}, the excited diagrams
in this case are in bijection with flagged tableaux of shape
$\delta_k$ with entries in row~$i$ at most $\vt_i = k/2+i$. By
subtracting~$i$ from the entries in row~$i$ of such tableaux we obtain
reverse plane partitions of shape $\delta_k$ with entries at most~$k/2$,
as desired.
\end{proof}

\begin{remark}
Proctor's formula also counts the number of plane partitions inside
a $[k\times k \times k]$ cube whose (matrix) transpose is the same as its
complement~(see also~\cite[Class 6]{Kra-plane}, \cite[Case~6]{Stanley_sym} and
\cite[\href{http://oeis.org/A181119}{A181119}]{OEIS}). This is because
the anti-diagonal consists of the values~$k/2$, and the plane partition
below of the diagonal is of shape $\delta_k$ with entries at most~$k/2$.
Clearly, this determines the rest of the partition.
\end{remark}

From the lemma we obtain:
$$
\log \ed(\ups_k) \, \sim \, \left(\frac{3\log 3}{2} \. - \. 2 \log 2 \right)n \. + \. o(n)
=\, 0.2616 \ts n\. + \. o(n)\ts,
$$
and therefore
$$
\aligned
\log e(\ups_k) \, & \le \,
\log F\bigl(\ups_k\bigr) \. + \. \log \ed\bigl(\ups_k\bigr) \,  = \,
\frac{1}{2}\ts n \log n \. + \left(\frac16 \ts - \ts
\frac{7\ts\log 2}{2} \ts + \ts 2\ts\log 3 \right)  n + \. o(n) \\
& \le \,
\frac{1}{2}\ts n \log n \. - \. 0.0621 \ts n\. + \. o(n) \ts.
\endaligned
$$
This proves the upper bound in Theorem~\ref{t:ribbons}.

\subsection{Comparison with general bounds} \label{ss:ribbons-gen}
Denote by $\cP_k$ the poset corresponding to $\ups_k$.  We have $\lan(\cP_k)= 2k-1$,
$\lc(\cP_k)=k$, and $r_1=k$, \ldots, $r_k=2k-1$.
Consider partition of the poset $\cP_k$ corresponding to $\ups_k$ into chains
of lengths $1,2,\ldots,k-1$ and $k$ chains of length~$k$.  Now
Theorem~\ref{t:poset-gen} gives
$$
r_1!\cdots r_k! \, = \, \frac{\Phi(2k-1)}{\Phi(k-1)} \, \le
e(\ups_k) \, \le \, \binom{n}{1,2,\ldots, k-1,k,k,\ldots,k} \, = \, \frac{n!}{(k!)^{k-1} \ts \Phi(k)}\..
$$
In other words,
$$
\aligned
\log e(\ups_k) \. & \ge \.
\log \frac{\Phi(2k-1)}{\Phi(k-1)} \. =
\. \frac{1}{2}\ts n \log n \. + \left(\frac{11\ts\log 2}{6} \ts - \ts \frac{\log 3}{2} \ts - \ts \frac{3}2\right) n + \. o(n) \\
& \ge
\. \frac{1}{2}\ts n \log n \. - \. 0.7785 \ts n \. + \. o(n) \ts, \ \ \text{and}\\
\log e(\ups_k) \. & \le \. \log \frac{n!}{(k!)^{k-1} \ts \Phi(k)} \.
= \. \frac{1}{2}\ts n \log n \. + \left(\frac16 \ts - \ts \frac{\log 2}2
\ts + \ts \frac{\log 3}{2} \right) n + \. o(n) \\
& \le
\. \frac{1}{2}\ts n \log n \. + \. 0.3694 \ts n \. + \. o(n) \ts.
\endaligned
$$
In notation of conjecture in $\S$\ref{ss:jaybounds}, denote by $c$ the (conjectural)
constant in the asymptotics
$$\log e(\ups_k) \, =  \, \frac{1}{2}\. n \ts \log n \. + \. c \ts n \. + \. o(n)\ts.
$$
Our bounds imply that $c\in [-0.3237, -0.0621]$. This is much sharper than the bounds
$c\in [-0.7785,  0.3694]$ which follows from Theorem~\ref{t:poset-gen}.

Finally, by the HLF for $e(\de_k)$ (see
\cite[\href{http://oeis.org/A005118}{A005118}]{OEIS}), we have:
$$\aligned
\log e(\de_k) \, & = \, \log \frac{\binom{k}{2}!}{\La(k)} \, = \,  \,  \frac{1}{2}\. n \ts \log n \. + \. \left(\frac12 - \frac32 \ts\log 2\right) \ts n \. + \. o(n) \\
 & =  \,  \frac{1}{2}\. n \ts \log n \. - \. 0.5397 \ts n \. + \. o(n) \ts,
\endaligned
$$
where $n=|\de_k|=\binom{k}{2}$.  Proposition~\ref{p:skew-bound} then gives:
$$\aligned
\log e(\ups_k) \, & \le  \, \log \frac{|\de_{2k}|!\.e(\de_k)}{|\de_{k}|!\.e(\de_{2k})}
 \, =  \,  \frac{1}{2}\. n \ts \log n \. + \. \left(\frac32 - \frac12 \ts\log 2\right) \ts n \. + \. o(n) \\
 & \le  \,  \frac{1}{2}\. n \ts \log n \. + \. 1.1534 \ts n \. + \. o(n) \ts.
\endaligned
$$
Note that this bound has correct first asymptotic term once again, but is weak in the $O(n)$ term.

\medskip

\subsection{Application to Littlewood--Richardson coefficients}\label{ss:ex-LR}
Let us show how the upper bounds in Theorem~\ref{t:ribbons} imply new
bounds on the LR--coefficients (cf.\ the proof of Proposition~\ref{p:skew-bound}).
Denote by~$\nabla$ the (triangular) strongly stable shape of staircase shapes $\de_k$.

\begin{corollary} \label{c:LR}
Let $\la^{(k)} = \de_{2k}$, $\mu^{(k)}= \de_k$, and $n=|\la/\mu| = 3k(k+1)/2$ as above.
Suppose $\nu^{(k)}\vdash n$ has strongly stable shape~$\nabla$. Then
$$\log \. c^{\la^{(k)}}_{\mu^{(k)},\.\nu^{(k)}} \, \le  \,
\left(2 \ts \log 3 \. - \. \frac13 \. - \. 2 \ts\log 2\right) \ts n \. + \. o(n)\ts.
$$
\end{corollary}

The constant in the RHS is $\approx 0.4776$.  It is unlikely to be sharp, but
is a rare explicit result currently available in the literature
(see~\cite{Na} and~$\S$\ref{ss:fin-lower}).

\begin{proof}  Recall that $\ts c^\la_{\mu\ts\nu}\le f^{\la/\mu}/f^{\nu}$,
see the proof of Proposition~\ref{p:skew-bound}.  Now apply the upper bound
for $e(\ups_k)$ in Theorem~\ref{t:ribbons} and the asymptotics for
$e(\de_r)$ given above, with $r=\sqrt{2n/3}$.
\end{proof}

\begin{remark}
Note that corollary likely extends to
all three shapes having strongly stable shape~$\nabla$, if one could
obtain an asymptotic version of Lemma~\ref{l:ribbons-excited}.
\end{remark}

\medskip

\subsection{Thick ribbons of subpolynomial depth}\label{ss:ex-thick}
Consider now thick ribbons $\ups_k=\de_{k+g(k)}/\de_k$,
where $g(k)$ is a subpolynomial growth function.
Note that $n=|\ups_k|=k\ts g(k) +O(g(k)^2)$.

\begin{theorem}\label{t:ex-thick-subpoly}
Let $g(k)$ be a subpolynomial growth function, and let $\{\ups_k\}$ be thick ribbons
defined as above.  Then
$$-\log 2 \. + \. o(1)\, \le \,
\frac{1}{n}\Bigl(\log e(\ups_k) \, - \, n \log n \. + \. n \ts \log g(k)\Bigr)
\, \le \, 0\. + \.o(1) \quad \text{as} \ \  k \to \infty\ts,
$$
where \ts $n=k\ts g(k)$.
\end{theorem}

\begin{proof}
By a direct computation:
$$\aligned
\log F(\ups_k) \, & = \, n \ts \log n \. - \. \log\bigl((2g(k)-1)!!\bigr)^k
\. - \. O\bigl(g(k)^2 \log g(k)\bigl)  \\
\, & = \, n \log n \. - \. n \log g(k) \. - \. n \log 2\. - \. O\bigl(g(k)^2 \log g(k)\bigl)\ts.
\endaligned
$$
The result follows from Theorem~\ref{t:main} and Lemma~\ref{l:exp}.
\end{proof}

\begin{remark}\label{r:examples-log2}
We conjecture that the lower bound is tight in the theorem.  This
is supported by the zigzag ribbon shapes calculations (see below).
Note also that Proctor's formula applies to all thick ribbons $\de_{k+2r}/\de_k$.
It gives $\ts\ed(\ups_k) = 2^{n(1+o(1))}$ for $r=g(k)$ of subpolynomial growth.

For comparison, in notation of Theorem~\ref{t:ex-thick-subpoly},
the lower bound in Theorem~\ref{t:poset-gen} gives a weaker lower bound of $\ts -1 + o(1)$.
Curiously, the upper bound
$$
e(\ups_k)\, \le \, \frac{n!}{\bigl(g(k)!\bigr)^{k-1}\ts \Phi\bigl(g(k)\bigr)}
$$
given by Theorem~\ref{t:poset-gen} matches the upper bound in
Theorem~\ref{t:ex-thick-subpoly}.
% improves the upper bound on $f^{\la/\mu}$ implied by Theorem~\ref{t:poset-gen}.
\end{remark}

\bigskip

\section{Two dual shapes}

\subsection{Inverted hooks} \label{e:posets-hook}
Consider the \emph{inverted hook shape} \ts $\ups_k=(k+1)^{k+1}/k^{k}$,
(cf.~$\S$\ref{ss:ex-thick} and~\cite[$\S$3.1]{MPP1}).
We have $n=|\ups_k|=2k-1$ in this case.  Straight from the definition, we have:
$$
e(\ups_k) \. = \. \binom{2k}{k}\., \qquad F(\ups_k) \. = \,
\frac{(2k+1)!}{(k+1)!^2} \, = \, \frac{2k+1}{(k+1)^2} \binom{2k}{k}.
$$
Since $\binom{2k}{k}\sim c\ts 4^k/\sqrt{k}$, the lower bound in Theorem~\ref{t:main} is
off only by a $O(k)$ factor. The lower bound in Theorem~\ref{t:poset-ineq}
clearly coincides with this bound for a poset $\cP$ corresponding to~$\ups_k$
and is the exact for the dual poset~$\cP^\ast$.  At the same time, the lower bound in
Theorem~\ref{t:poset-gen} is~$2^k$, which is off by an exponential factor.

For the upper bound, observe that the excited diagrams are exactly
the complements of the paths~$\ga$ from the SW to the NE~corner
of the $(k+1)$-square~$\la^{(k)}$.
This gives \ts $\ed(\ups_k) \ts = \ts \binom{2k}{k}$,
and the upper bound in Theorem~\ref{t:main} gives only $O(16^k/k^2)$
(cf.~$\S$\ref{ss:examples-box}). On the other hand, the upper bound in
Theorem~\ref{t:poset-gen} is $\binom{2k-1}{k}$, which is off only by a factor of~2.

\medskip

\subsection{Inverted thick hooks}\label{ss:examples-box}
Let $\vk_k=(2k)^{2k}/k^{k}$. We call~$\ups_k$ \emph{inverted thick hooks};
they have strongly stable shape (see Figure~\ref{f:ribbon}).
The following is the analogue of Lemma~\ref{l:ribbons-excited}.

\begin{proposition} \label{p:inverted-thick}
For $\vk_k=(2k)^{2k}/k^{k}$, we have:
$$
\ed\bigl(\vk_k\bigr) \, = \, \prod_{i=1}^k \. \prod_{j=1}^k\. \frac{k+i+j-1}{i+j-1}
\, = \, \frac{\HH(k-1)^3\. \HH(3k-1)}{\HH(2k-1)^3}\,
\sim \, C \ts \left(\frac{3\sqrt{3}}{4}\right)^{n} n^{-1/24}\,,
$$
where $n=|\vk_k|=3k^2$ and $C>0$ is a constant.
\end{proposition}

As in the proof of Lemma~\ref{l:ribbons-excited},
the exited diagrams in this case are in bijection with \emph{solid partitions}
which fit inside a $k\times k \times k$ box, see~\cite{MPP2}.
The classical \emph{MacMahon's formula} (see e.g.~\cite{EC2} and
\cite[\href{http://oeis.org/A008793}{A008793}]{OEIS}),
gives the proposition.  We omit the details.

Now, by the calculation similar to the one in the previous section,
Theorem~\ref{t:stable-strong} gives the following lower and upper bounds:
$$
\aligned
\log e(\vk_k) \. & \ge \.
\log F\bigl(\vk_k\bigr) \. = \. \frac{1}{2}\ts n \log n \. - \.  n \. - \.
c_1 \ts \frac{n}3 \. - \. 2\ts c_2 \ts \frac{n}3 \. + \. o(n) \\
& \ge \. \frac{1}{2}\ts n \log n \. - \. 1.4095 \ts n \. + \. o(n) \ts, \\
\log e(\vk_k) \. & \le \.
\log F\bigl(\vk_k\bigr) \. + \. \log \ed\bigl(\vk_k\bigr)
 \. = \.
\log F\bigl(\vk_k\bigr) \. + \. n\ts \log \frac{3\sqrt{3}}{4} \. + \. o(n) \\
& \le \.
\frac{1}{2}\ts n \log n \. - \. 1.1479 \ts n\. + \. o(n) \ts.
\endaligned
$$
Here
$$
\aligned
c_1 \, & = \, \iint_{\cC_0} \log \ts (x+y) \, dx \. dy \, = \, 2 \ts \log2 \. - \. \frac34\, \approx \. -\ts 0.1137\ts,
\\
c_2 \, & = \, \iint_{\cC_0} \log \ts (1+x+y) \, dx \. dy \, = \, \frac92\ts \log 3  \. - \. 4 \ts \log2 \. - \.\frac32 \, \approx \. 0.6712\ts,
\endaligned
$$
and $\cC_0=[0,1]^2$ is the unit square.  On the other hand,
$$
\log e(\vk_k) \, = \, \frac{1}{2}\ts n \log n \. - \.  n \. - \.
2\ts c_1 \ts \frac{n}{3} \. - \. \ts c_3 \ts \frac{n}{3} \. + \. o(n)\, =
\, \frac{1}{2}\ts n \log n \. - \. 1.2873 \ts n  \. + \.  o(n)\ts,
$$
where
$$c_3 \, = \, \iint_{\cC_0} \log \ts (2+x+y) \, dx \. dy \, = \, 18 \ts \log2
\. - \. 9\ts \log 3 \. - \. \frac32\, \approx \. 1.0891\ts.
$$
In other words, the correct asymptotics in this case is near the middle
between the upper and lower bounds in Theorem~\ref{t:stable-strong}.
%\end{example}

Similarly, Theorem~\ref{t:poset-gen} implies the following
upper and lower bounds:
$$\aligned
e(\vk_k) \, & \ge \,
2!\ts 4! \ts \cdots (2k)! \ts (2k-1)! \cdots 2! \ts 1! \, = \, \frac{\Phi(2k)\ts \Phi(2k-1)}{\Psi(k)}\,, \\
e(\vk_k) \, & \le \, \binom{n}{3k-1,3k-2,3k-4,3k-5,\ldots, 5,4,2,1} \. =
\, \frac{n! \cdot 3!\ts 6!\ts \cdots \ts (3k)!}{\Phi(3k)}\,,
\endaligned
$$
where the upper bound follows from the chain decomposition as in Figure~\ref{f:ribbon}.
Using the asymptotics in $\S$\ref{ss:posets-not}
and~\cite[\href{http://oeis.org/A268504}{A268504}]{OEIS}, this gives:
$$
\aligned
\log e(\vk_k) \. & \ge \.
\frac{1}{2}\ts n \log n \. + \left(-\frac32 \ts - \ts \frac{\log 3}2+\log 2\right) n + \. o(n) \, =
\, \frac{1}{2}\ts n \log n \. - \. 1.3562 \ts n \. + \. o(n)\ts,\\
\log e(\vk_k) \. & \le \.
\frac{1}{2}\ts n \log n \. + \left(\frac12 \ts - \ts \frac{\log 3}{2}\right) n + \. o(n)  \,
= \, \frac{1}{2}\ts n \log n \. - \. 0.0493 \ts n \. + \. o(n) \ts.
\endaligned
$$

In summary, for the constant $c=-1.2873$ in second term of the asymptotics of $e(\vk_k)$,
Theorem~\ref{t:stable-strong} gives $c\in [-1.4095, -1.1479]$.  Similarly,
Theorem~\ref{t:poset-gen} gives $c\in [-1.3562, -0.0493]$ which is much wider,
but has a slightly better lower bound.  This is not a surprise: the further
skew shape is from the regular shape, the larger is the number $\ed(\vk_k)$ of
excited diagrams, and thus the worse are the bounds given by Theorem~\ref{t:main}
(see Example~\ref{e:posets-hook} for an extreme special case).

Note that if we use a weaker bound $\log \ed(\vk_k)\le (\log 2)n$ given by Lemma~\ref{l:exp},
we get $c \le -0.7164$, which is still sharper than the bound
$c\le -0.0493$ from Theorem~\ref{t:poset-gen} (cf.\ Remark~\ref{r:examples-log2}).

\begin{remark}\label{r:examples-sym}
There is a strong connection between the cases of thick hooks and thick ribbons,
since excited diagrams in both cases enumerate certain families of solid partitions in a box
(equivalently, lozenge tilings of a regular hexagon).  We refer to~\cite{Stanley_sym}
for an early overview and to~\cite{Kra-plane} for a recent survey.  Now, it may seem puzzling
that $\ts \log \ed(\vk_k) \sim 2\log \ed(\ups_k)\ts $ in these examples.  In one direction this
is straightforward --- simply observe that $\ed(\vk_k) \ge \ed(\ups_k) \ed(\ups_{k+1})$
(see Figure~\ref{f:ribbon}).
\end{remark}

\begin{figure}[hbt]
\includegraphics[width=12.3cm]{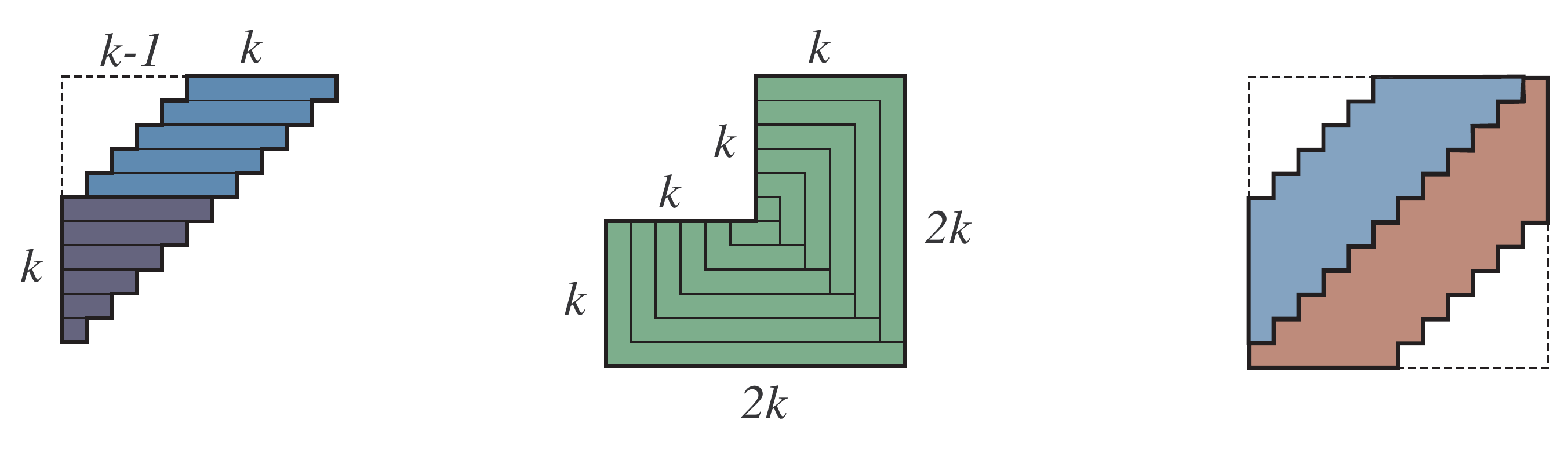}
\caption{Chain partitions of the thick ribbon shape $\ups_k$
and inverted thick hook $\vk_k$, for $k=6$. Two thick ribbons
forming an excited diagram of the inverted thick hook.}
\label{f:ribbon}
\end{figure}

\bigskip

\section{Ribbon shapes} \label{s:ribbon}

\subsection{Zigzag shapes}\label{ss:ex-zigzag}
Let $\Alt(n)= \{\si(1)<\si(2)>\si(3)<\si(4)>\ldots\} \ssu \SS_n$ be the set of
{\em alternating permutations}.
The number $E_n=|\Alt(n)|$ is the $n$-th {\em Euler number} (see~\cite{Stanley_SurveyAP}
and \cite[\href{http://oeis.org/A000111}{A000111}]{OEIS}), with the asymptotics
% g.f.
% $$
% 1\, + \,\sum_{n=1}^\infty \. E_n \. \frac{z^n}{n!} \,\. = \,\, \tan(z) \ts + \ts \sec(z)\..
% $$
% Recall
$$
E_n \, \sim \, n!\. \left(\frac{2}{\pi}\right)^n\frac{4}{\pi}\.  \bigl(1+ o(1)\bigr) \quad \text{as} \ \ n\to \infty
$$
(see e.g.~\cite{FS,Stanley_SurveyAP}).
%
% Let $\delta_k =(k-1,k-2,\ldots,2,1)$ denotes the \emph{staircase shape} and
Consider the \emph{zigzag ribbon hook} $\rho_k=\de_{k+2}/\de_k$, $n=|\rho_k|=2k+1$
(see Figure~\ref{f:curve}).
% This ribbon hook is sometimes called the \emph{zigzag ribbon}.
% (see Figure~\ref{f:curve}).
Clearly,
$e(\rho_k)=E_{2k+1}$.
% It is well known and easy to see that $\ts e(\rho_k) \ge e(\tau)\ts$
% for every ribbon hook $\tau$ with $|\tau|=n$.
%
Observe that $\ed(\rho_k)=C_{k+1}$, the $(k+1)$-st Catalan number
$$
 C_m \, = \, \frac{1}{m+1}\binom{2m}{m} \. \sim \. 4^m \ts m^{-3/2} \ts \pi^{-1/2} %\..
$$
(see~\cite{MPP2}).
Thus $\ts \ed(\rho_k) = \Theta\bigl(2^n/n^{3/2}\bigr)$,
so Lemma~\ref{l:exp} is asymptotically tight for the staircase shape.
% Also, $d(\rho_k) \sim k/2$, so Lemma~\ref{l:poly} is weaker than
% Lemma~\ref{l:exp} in this case.

Writing \ts $\ed(\rho_k) \ts \sim \ts c\ts n! \ts \eta^n \ts n^\al$,
let us compare the estimates from different bounds. We have
$\ts F(\rho_k)= n!/3^k$, so Theorem~\ref{t:main} for $\rho_k$ gives
$$
% (\triangle) \qquad
\frac{n!}{3^k} \, \le \, E_n \, \le \,  \frac{n!\cdot C_k}{3^k}\.,
$$
implying that $1/\sqrt{3} \le \eta \le 2/\sqrt{3}$.  Note that
$\ts 1/\sqrt{3} \approx 0.577$, $2/\pi \approx 0.636$ and $2/\sqrt{3} \approx 1.155$\ts.
This example shows that the lower bound in Theorem~\ref{t:main} is nearly tight,
while the upper bound is vacuous for large~$n$ (we always have $e(\cP)\le n!$, of course).

Similarly, Theorem~\ref{t:poset-gen} for $\rho_k$ gives
$$
% (\lozenge) \qquad
k!\ts (k+1)! \, \le \, E_n \, \le \,  \frac{n!}{2^k}\,,
$$
implying $0.5 \le \eta \le 1/\sqrt{2} \approx 0.707$.  Thus,
in this case the general poset lower
bound is weaker, while the upper bound is much sharper
than $(\ast)$ in Theorem~\ref{t:main}.

\begin{remark}\label{r:zigzag}
The asymptotics of $e(\la/\mu)$ is known for other ribbons with a
periodic pattern (see Figure~\ref{f:curve}).
For example, for \emph{two up, two right ribbons}~$\rho_n$ we have
\ts $e(\rho_n)= \Theta(n! \ts \al^{n})$, where
$\al \approx 0.533$ is the smallest positive root of
$\ts\bigl(\cos \frac{1}{\al}\bigr) % \ts % \cdot
\bigl(\cosh \frac{1}{\al}\bigr) =-1$,
see~\cite{CS} and~\cite[\href{http://oeis.org/A131454}{A131454}]{OEIS}.
It would be interesting to find the asymptotics for more general
ribbons, e.g.\ ribbons along a curve as in the figure.
\end{remark}

\begin{figure}[hbt]
\epsfig{file=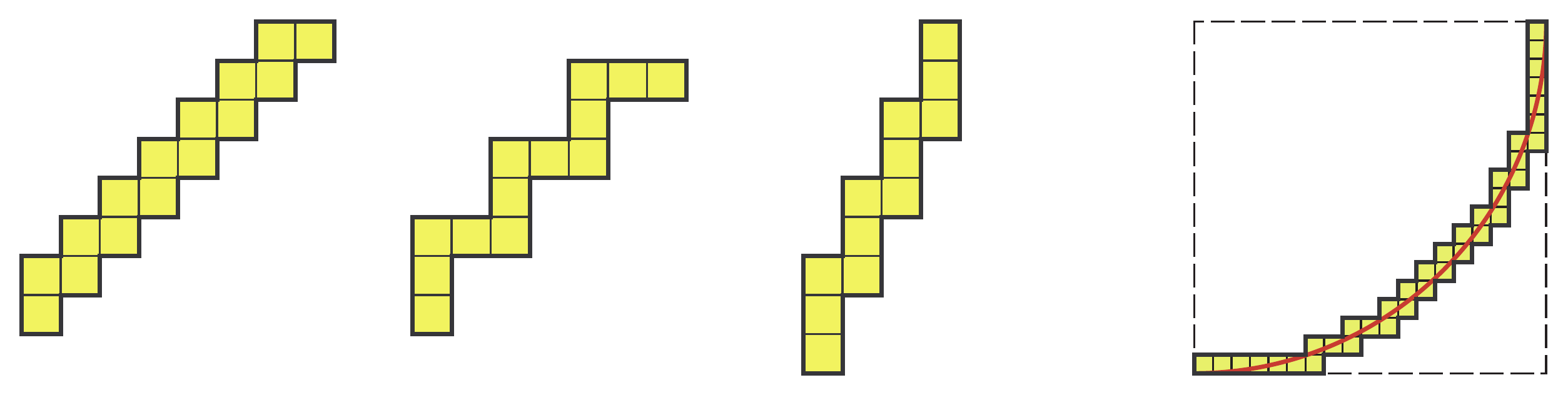, width=10.2cm}
\caption{Zigzag ribbon $\rho_7=\de_9/\de_7$, ``two up two right'' ribbon, ribbon
$\rho(4,3)$, and a ribbon along the circle segment. }
\label{f:curve}
\end{figure}

\medskip

\subsection{Ribbons with subpolynomial depth}\label{ss:ribbon-k-up}
Let $n=k\ts m$, and let $\rho(k,m)$ be the unique ribbon hook
where all columns have length~$m$ (see Figure~\ref{f:curve}).
% For $m=2$ this shape is similar to the zigzag shape~$\rho_n$.
We have $e\bigl(\rho(k,m)\bigr)$ is the
number of ``$(m-1)$ up, $1$ down'' permutations in~$S_n$.

% For fixed~$m$, both the g.f.\ and the asymptotics for $a(k,m)$
% are well understood (see e.g.~\cite{GJ,Stanley_SurveyAP}).
% There is also a formula for $\ed(\rho_n)$ in this case in terms of the
% \emph{Fuss-Catalan numbers}:
% $$(\circledcirc) \qquad \
% \ed(\rho_n) \, = \, \frac1{n-k+1}\binom{n}{k}\, = \, \frac1{(m-1)\ts k+1}\binom{m\ts k}{k}\ts.
% $$
% (see e.g.~\cite{GJ,StCat}).

Consider the case $m=g(k)$, where $g(k)$ is a subpolynomial function.
The ribbon $\rho(k,m)$ has a subpolynomial depth, so
Theorem~\ref{t:sub-poly} applies.
% $$\log a\bigl(k,m\bigr) \, = \, n \log n \. - \. O(n \log m)\ts.
%$$
The following result gives a sharper estimate (cf.~$\S$\ref{ss:ex-thick}):

\begin{theorem} \label{t:ribbons-subpoly}
Let $n=k\ts g(k)$, where $g(k) \to \infty$, $g(k) = k^{o(1)}$
as $k \to \infty$.  Then:
$$\log e\bigl(\rho(k,g(k))\bigr) \, = \, n \log n \. - \. n \ts \log g(k) \. - \.
\frac{n \ts \log g(k)}{2 \ts g(k)} \. - \.  \frac{n}{g(k)}
\. + \. O\left(\frac{n}{g(k)^2}\right) \quad
\text{as} \quad n \to \infty\ts.
$$
\end{theorem}

\begin{proof}  Denote $a(k,m) = e\bigl(\rho(k,m)\bigr)$.
By the lower bound in Theorem~\ref{t:main}, we have:
$$a\bigl(k,g(k)\bigr) \, \ge \,
F(\tau_n) \, \sim \, \frac{n!}{(m-1)!^k(m+1)^k} \qquad \text{as} \quad k \to \infty\ts,
$$
where $m=g(k)$.
On the other hand, by the upper bound in Theorem~\ref{t:poset-gen},
we have:
$$
a\bigl(k,g(k)\bigr) \, \le \, \binom{n}{m, \ts m, \ldots, \ts m} \, = \, \frac{n!}{m!^k} \ts.
$$
We conclude:
$$\aligned
\log a\bigl(k,g(k)\bigr) \, & = \, \log F(\tau_n) \. + \. k \ts O\bigl(\log (1+1/m)\bigr) \, =
\, n \log n \. - \. k \ts \log m! \. + \. O(k/m)\\
\, & = \, n \log n \ts -\ts n \ts - \ts O(\log n) \. - \.
k\ts \left[m \log m \ts - \ts m \ts + \ts \frac12\ts \log m + 1 + O\left(\frac1m\right) \right]  \. + \. O(k/m) \\
& = \, n \log n \. - \. n \ts \log g(k) \. - \. \frac{n \ts \log g(k)}{2 \ts g(k)} \. - \.  \frac{n}{g(k)}
\. + \. O\left(\frac{n}{g(k)^2}\right)\ts,
\endaligned
$$
as desired.
\end{proof}

\bigskip

\section{Slim shapes}\label{s:slim}

\subsection{Number of excited diagrams} \label{ss:ex-slim}
Let $\la= (\la_1,\ldots,\la_\ell)$ be such that $\la_\ell \ge \mu_1+\ell$.
Such partitions are called \emph{slim}.
The following result improve the bound in Lemma~\ref{l:poly} in this case.

\begin{proposition}\label{p:ex-comp-rectangle}
In the notation above, let $\ts\mu\vdash m\ts$ be fixed and let
$\ts \bigl\{\la^{(k)}\bigr\}\ts $ be a
family of slim partitions, such that \ts
$\ell(k):=\ell\bigl(\la^{(k)}\bigr) \to \infty$ \ts as \ts $k \to \infty$.
We have:
$$
\ed\bigl(\la^{(k)}/\mu\bigr) \, \sim \, \frac{\ell(k)^m}{\prod_{x\in \mu} \, h(x)} \quad \text{as} \quad k \to \infty\..
$$
\end{proposition}

Note that there no direct dependence on $n=\bigl|\la^{(k)}/\mu\bigr|$ in this case.

\begin{proof}
We think of every finite $S\ssu \nn^n$ as a poset with order relations given by
$(x,y) \preccurlyeq (x',y')$ if $x\le y$ and $y\le y'$.

Consider random subset $S$ of squares of $\mu$ on the corresponding diagonals
inside~$\la$, which all have lengths between \ts $\ell-O(1)$ \ts and~$\ell$.
The subset $S\subseteq \mu$ is an excited diagram if and only in its poset
is a refinement of the poset~$\mu$. In one direction this follows by
induction and the definition of excited diagrams (order relations in $\mu$
do not disappear under moves). In the other direction, this is more delicate
and follows from~\cite[$\S$3.3]{MPP1}.

Now, observe that for the points sampled from the unit interval,
the probability the order relations are satisfied is exactly
$$\frac{|\SYT(\mu)|}{m!} \, = \, \prod_{x\in \mu} \. \frac{1}{h(x)}\..
$$
Indeed, this probability is equal to the volume of the \emph{poset polytope}
(see e.g.~\cite{BR,EC2}).  From above, the number of such diagrams is \ts
$\bigl(\ell(k)-O(1)\bigr)^m$, which implies the result.
\end{proof}

\medskip

\subsection{Complementary shapes in large slim rectangles}\label{ss:ex-slim-OO}
The case of a fixed $\mu$ and $\la$ a rectangle
is especially interesting (cf.~$\S$\ref{ss:ex-comp} below).

\begin{theorem}[\cite{RV}] \label{t:slim-RV}
Let $\mu$ be fixed and $\la=(k^\ell)$ with $\ell, k/\ell\to \infty$.  We have:
$$
\frac{e(\la/\mu)}{e(\la)}
\, \sim \, \prod_{x\in \mu} \.
\frac{1}{h(x)}\,, \quad \text{as} \ \ \ell, k/\ell \to \infty\ts.
$$
\end{theorem}

This theorem is obtained by Regev and Vershik in~\cite{RV} by a calculation based
on~\cite{OO}.  Let us show that the upper bound in the theorem follows from the
upper bound in Theorem~\ref{t:main}.

\smallskip

Let $m=|\mu|$. We have in this case:
$$
\frac{F(\la/\mu)}{e(\la)} \, = \, \frac{(k\ell-m)!}{(k\ell)!} \. \left[\prod_{x\in \mu} \. h_\la(x)\right],
$$
where $h_\la(x)$ is a hook-length in~$\la$. Letting $k/\ell \to \infty$, we get
$$
\frac{F(\la/\mu)}{e(\la)}
\, \sim \, \frac{(k\ell-m)! \. k^m}{(k\ell)!} \, \sim \, \frac{1}{\ell^m}\,.
$$
Now the upper bound in Theorem~\ref{t:main} and Proposition~\ref{p:ex-comp-rectangle}
give:
$$
\frac{e(\la/\mu)}{e(\la)}
\, \lesssim \, \frac{\ed(\la/\mu) \ts F(\la/\mu)}{e(\la)} \, \sim \, \prod_{x\in \mu} \.
\frac{1}{h(x)}\,, \quad \text{as} \ \ \ell, k/\ell \to \infty\ts.
$$
This implies that the upper bound in Theorem~\ref{t:main}
is asymptotically tight in this case, while the lower bound
is off by a \ts $\Theta\bigl(\ell(\la)^{|\mu|}\bigr)$ \ts factor.

\medskip

\subsection{Slim stripes} \label{ss:slim-stripes}
Let $\ups=\la/\mu$ be slim skew shape defined above.
It follows from the proof of Proposition~\ref{p:ex-comp-rectangle}, that
the number of excited diagrams $\ed(\la/\mu)$ in this case depends only
on~$\mu$ and~$\ell$, but not on~$\la$.  The following result considers
a special case of a staircase shape $\mu=\de_\ell$.

\begin{proposition}
Let $\la/\mu$ be a skew shape s.t.
$\mu=\de_\ell$, $\ell(\la)=\ell$ and $\la_\ell \ge \mu_1+\ell$.
Then \ts $\ed(\la/\mu) = 2^{\binom{\ell}{2}}$.
\end{proposition}

For example, for $\la/\mu=(765/21)$, we have $\ell=3$ and
$\ed(\la/\mu)=8$.

\begin{proof}
By the proof of Theorem~\ref{t:ed-det} in~\cite{MPP1}, the
set of corresponding flagged tableaux in this case is in
bijection with $\SSYT$s of shape $\de_\ell$ with all
entries~$\le \ell$.  In general these constraints vary
from row to row, but for $\mu=\de_\ell$ and slim~$\la$
as above, we have $\ssf_i=\ell$ for all $1\le i \le \ell$.

Now, the number of $\SSYT$s of shape $\mu$ with all
entries~$\le \ell$ is equal to the value of Schur function
$s_\mu(1,\ldots,1)$, $\ell$ times, which can be computed by
the \emph{hook--content formula}~\cite[Cor.~7.21.4]{EC2}.
A direct calculation gives $\ts s_{\de_\ell}(1,\ldots,1)=2^{\binom{\ell}{2}}$.
\end{proof}

\begin{remark}
There is a curious bijection between excited diagrams in this case
and domino tilings of an Aztec diamond $AD_\ell$, since $\SSYT$s of
shape $\de_\ell$ with all entries~$\le \ell$ have Gelfand--Tsetlin
patterns given by \emph{complete monotone triangles} of size~$\ell$.
In turn, the latter are in bijection with domino tilings of~$AD_\ell$,
see~\cite{EKLP}.
\end{remark}

\bigskip

\section{Further examples and applications} \label{s:ex}

\subsection{Dual shapes} \label{ss:ex-comp}
Let $\la = (k^\ell)$ be a rectangle and $\mu \ssu \la$. Denote by
$\nu = (\la/\mu)^\ast$ the partition obtained by the
180 degree rotation of $\la/\mu$.
Thus we can apply the lower bound in Theorem~\ref{t:main} to obtain the
following unusual result.

\begin{proposition}
Let $\nu$ be a partition of~$n$.  Denote by $h^\ast(x)$, $x=(i,j)\in \nu$,
the \emph{dual hooks} defined by $h^\ast(i,j) = i+j-1$, and let
$$
H(\nu) \, = \, \prod_{x \in \nu} \. h(x)\ts, \qquad
H^\ast(\nu) \, = \, \prod_{x \in \nu} \. h^\ast(x)\ts.
$$
Then \ts
$H(\nu) \le H^\ast(\nu)$, and the equality holds only
when $\nu$ is a rectangle.
\end{proposition}

Note that this result cannot be proved by a simple monotonicity
argument, since
$$
\sum_{x\in \nu} \. h(x) \, = \, \sum_{x\in \nu} \. h^\ast(x) \,
= \, n \. + \. \sum_i \binom{\nu_i}{2} \. + \. \sum_j \binom{\nu_j'}{2}\..
$$

\begin{proof}
Observe that $e(\la/\mu) = e(\nu) = n!/H(\nu)$ and
$F(\la/\mu) = n!/H^\ast(\nu)$, where $n=|\la/\mu| = |\nu|$.
Now the inequality $e(\la/\mu)\ge F(\la/\mu)$ implies the first part.
For the second part, recall from the proof of Theorem~\ref{t:main}
that $e(\la/\mu) = F(\la/\mu)$ only if $\ed(\la/\mu)=1$, i.e.~when
no excited moves are allowed.  It is easy to see that this can
happen only when~$\nu$ is a rectangle.
\end{proof}

\begin{remark}
After the proposition was obtained, in response to second author's
{\tt MathOverflow} question, F.~Petrov found a generalization of
this result to all concave functions.\footnote{See \ts
\href{http://mathoverflow.net/q/243846}{http://mathoverflow.net/q/243846}.}
In particular, his proof implies that the variance of hooks is larger
than that of complementary hooks in all Young diagrams.
\end{remark}

\medskip

\subsection{Regev--Vershik shapes}
Let $\sigma \ssu \tau$, where $\tau = (\ell^k)$ is a rectangle.
As above, denote by $\si^\ast$ the
180~degree rotation of~$\si$.
Consider a skew shape $\la/\mu$ obtained by attaching two copies
of~$\si^\ast$ above and to the left of $\tau$, and removing
$\si^\ast$ from~$\tau$ (see Figure~\ref{f:R-V}).  The theorem by Regev and
Vershik in~\cite{RV} states that
$$(\diamond) \qquad \ \
\prod_{x\in \la/\mu} \. h(x) \, = \, \left[\prod_{x\in \si} \. h(x) \right]
\ts \cdot \ts \left[\prod_{x\in \tau} \. h(x) \right].
$$

\smallskip

\begin{proposition}
In the notation above, let $s:=|\si|$, $t:=|\tau| = k\ell$.  Then:
$$
e(\la/\mu)  \, \ge \, \binom{s+t}{s} \. e(\si) \ts e(\tau)\ts.
$$
\end{proposition}

The inequality is trivially tight for $\si=\varnothing$ or $\si = \tau$,
when the skew shapes coincide: $\la/\mu = \si \circ \tau$.

\begin{proof}  Note that $|\la/\mu|=|\si|+|\tau|=s+t$.
By the Regev--Vershik theorem~$(\diamond)$, we have
$$
\frac{F(\la/\mu)}{(s+t)!} \, = \, \frac{e(\si)}{s!} \ts \cdot \ts
\frac{e(\tau)}{t!}\,.
$$
By the lower bound in Theorem~\ref{t:main}, we have:
$$
e(\la/\mu) \, \ge \,  F(\la/\mu) \, = \,
\binom{s+t}{t} \. e(\si) \ts e(\tau)\.,
$$
as desired.
\end{proof}

\begin{remark}{\rm
Regev and Vershik conjectured that $\la/\mu$ and $\si\circ\tau$ have
the same multiset of hooks. This was proved bijectively and generalized
in a number of directions by  Janson, Regev and Zeilberger,
Bessenrodt, Krattenthaler, Goulden and Yong, and others.
We refer to~\cite{Kra} for the ``master bijection'' and
to~\cite[$\S$12]{AR} for further references.
}\end{remark}

\begin{figure}[hbt]
\includegraphics[width=12.8cm]{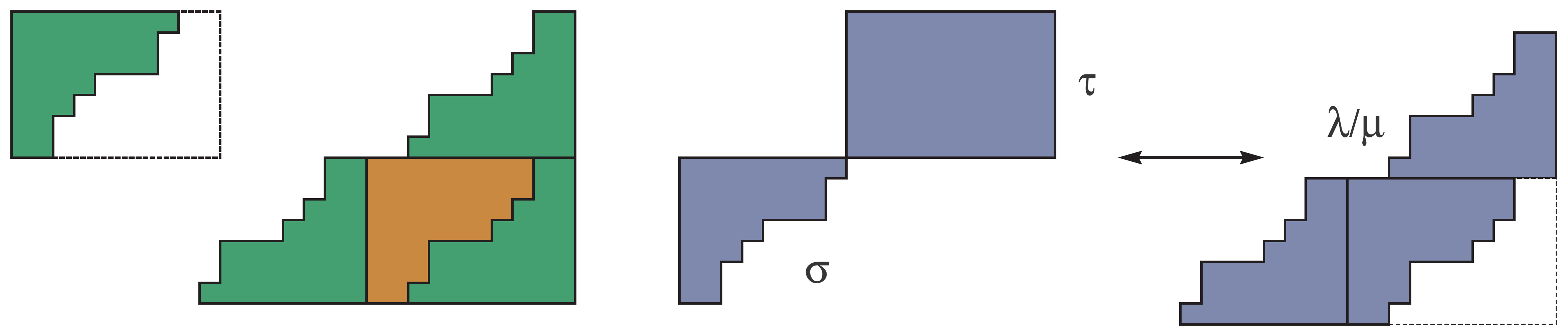}
\caption{Construction of the Regev--Vershik shape with $\si =(87^2432^2)$, $\tau=(10^7)$.
Now $(\diamond)$ follows from two skew shapes on the right having the same multisets of hooks. }
\label{f:R-V}
\end{figure}

\bigskip

\section{Final remarks and open problems} \label{s:fin}

\subsection{}\label{ss:fin-linear-extensions}
Computing $e(\cP)$ is known to be $\SP$-complete~\cite{BW}, which partly
explains relatively few good general bounds (see~\cite{ERZ,Tro}).
For a nice counterpart of the upper bound in Theorem~\ref{t:poset-gen},
relating $\log e(\cP)$ and the entropy, defined in terms of the $\lan(\cP)$,
see~\cite{C+,KK}.  Unfortunately this bound is not sharp enough for bounds
on $\ts \log e(\la/\mu)\ts$ as it is off by a multiplicative constant.

Note that this approach was used to improve the bounds on the second term
in the asymptotic expansion of $e(B_k)$, the number linear extensions of
the Boolean lattice~\cite{BT} (see also~\cite{SK,Bri}. Namely, for even~$k$,
Theorem~\ref{t:poset-gen} gives:
$$
\log_2 \binom{k}{k/2} \. - \.  \frac{3}{2} \ts \log_2 e \. + \. o(1) \, \le \,
\frac{\log_2 e(B_k)}{2^k} \, \le \,\log_2 \binom{k}{k/2}\ts.
$$
Brightwell and Tetali in~\cite{BT} show that the lower bound is tight.

\subsection{}\label{ss:fin-bounds}
We show that in many special cases of interest our bounds give the first or the
first two terms in the asymptotic expansion.  Although the number
$\ed(\la/\mu)$ of excited diagrams can be large, even exponential in many cases,
this is dwarfed by the number of standard Young tableaux, implying
that the ``naive HLF'' $F(\la/\mu)$ is indeed a good estimate in these cases.
% This also explains why Theorem~\ref{t:main} is not sharp for skew shapes with
% subexponential $e(\la/\mu)$ (cf.~$\S$\ref{ss:ex-slim}).

\subsection{}\label{ss:fin-other}
The arXiv version of~\cite{MPP1} contains an extended survey and comparison between
NHLF and other known formulas for $f^{\la/\mu}$.  Of course, the Jacobi--Trudy
and other determinant formulas (see e.g.~\cite{HG,LP,EC2}) are computationally efficient,
but difficult to use for asymptotic estimates.
While the Littlewood--Richardson rule is positive,
the LR--coefficients have a complicated combinatorial structure
and are hard to compute~\cite{Na} and estimate (cf.~$\S$\ref{ss:fin-lower} below).
Finally, the \emph{Okounkov--Olshanski
formula}~\cite{OO} mentioned earlier seems to be weak for large~$\mu$ (cf.~\cite{MPP3}.

\subsection{}\label{ss:fin-stable}
The notion of stable limit shape goes back to Erd\H{o}s and Szekeres
in the context of random partitions, and to Vershik and Kerov in the context of
Young tableaux and asymptotic representation theory (see e.g.~\cite{Bia,Rom}).
Our notion of strong stable shapes is more restrictive as we need faster convergence
for the proof of Theorem~\ref{t:stable-strong}.

\subsection{}\label{ss:fin-planar}
Note that the shifted skew diagrams greatly increases the number of limit
shapes of area~1, for which we have
$$
\log \ts
e(\ups_n) \, \sim \, \frac12 \. n \ts \log n \quad \text{as} \ \ n \to \infty\ts.
$$
As we mentioned in Remark~\ref{r:posets-trunkated}, this holds also for
some ``truncated shapes'' (see~\cite{AR,Pan}). We generalize this result
in a forthcoming~\cite{P2} to all piecewise linear shapes in the plane.

\subsection{}\label{ss:fin-lower}  Let us emphasize once again that it is the
lower order terms in the asymptotics of $f^{\la/\mu}$ that turn out to be most relevant
for applications (cf.~\cite{Bia,PR}).
For example, for the LR--coefficients we have:
$$(\circleddash) \qquad c^\la_{\mu\ts\nu} \, \le \, f^{\la/\mu}/f^{\nu}
$$
(see~$\S$\ref{ss:ex-LR}).  Following the proof of Corollary~\ref{c:LR},
in the stable shape case the leading terms for the RHS coincide since
$|\la/\mu|=|\nu|$, while the second order terms give an exponential upper bound
(see Theorem~\ref{t:stable-strong}).  In other words, any improvement in the
second order terms for $f^{\la/\mu}$ in every particular stable shape case
improves the upper bound~$(\circleddash)$ for the LR--coefficients.

Note that computing or even approximating the LR--coefficients is a major
problem in the area (see e.g.~\cite{Iken,Na}). In fact, the LR--coefficients
\ts $c^{\la}_{\mu,\ts \nu}$ \ts
are always at most exponential in $n=|\la|$, as recently
announced by Stanley in~\cite[Supp.~Exc.~7.79]{EC2}.  In the case
\ts $\la_1,\ell(\la) = O(\sqrt{n})$ \ts which includes the stable shapes in
Theorem~\ref{t:stable}, this easily follows from the \emph{Knutson--Tao puzzle}
interpretation of the LR--coefficients.

\subsection{} \label{ss:jaybounds}
For the thick ribbons $\ups_k = \delta_{2k}/\delta_k$, we conjecture that
$$\log f^{\ups_k} \. = \. \frac12 \ts\ts n \log n \. + \. c\ts n \. + \. o(n)\ts, \ \quad
\text{for some} \, c<0\ts.
$$
There seem to be no available tools to even approach this problem.
If true, the lower and upper bounds for $e(\ups_k)$ in Theorem~\ref{t:ribbons}
imply that \ts $-0.3237 < c < -0.0621$ (cf.~$\S$\ref{ss:ribbons-gen}).

Recently, Jay Pantone used his implementation
of the method of
{\em differential approximants} on $150+$ terms
of the sequence $\{e(\ups_k)\}$
\cite[\href{https://oeis.org/A278289}{A278289}]{OEIS} to approximate
the constant above as \ts $c \approx -0.1842$.\footnote{J.~Pantone,
personal communication (2016).} We plan to refine our techniques
to improve our bounds towards this value.

\subsection{}\label{ss:fin-shifted}
There is a shifted version of the NHLF also obtained by Naruse~\cite{Nar},
which should give similar asymptotic results for the number of SYT of
shifted shapes (cf.~\cite[$\S$5.3]{AR}).  In fact, the results should
follow verbatim for the stable shape limit case.
Let us also mention that our calculations
for thick ribbons (see Section~\ref{s:ribbons})
can be translated to this case; the excited diagrams in the shifted case
correspond to \emph{type~B} and have been extensively studied.  We refer
to~\cite{MPP2,MPP3} for details and further references.

%\newpage

\vskip.76cm

\subsection*{Acknowledgements}
We are grateful to Dan Betea, Stephen DeSalvo, Valentin F\'{e}ray,
Christian Krattenthaler, Nati Linial, Eric Rains,
Dan Romik, Bruce Sagan, Richard Stanley and
Damir Yeliussizov for helpful remarks on the subject. We thank Jay
Pantone for the calculation in $\S$\ref{ss:jaybounds}. The second author would
also like to thank Grigori Olshanski who introduced him to
the subject some decades ago, but that knowledge remained
dormant until this work.  The second and third authors were partially
supported by the~NSF. The first author is partially supported by an
AMS-Simons travel grant.

\vskip.9cm

% \newpage


\begin{thebibliography}{abcdef}

\bibitem[AR]{AR}
R.~Adin and Y.~Roichman, Standard {Y}oung tableaux, in
{\em Handbook of Enumerative Combinatorics} (M.~B\'ona, editor),
CRC Press, Boca Raton, 2015, 895--974.

\bibitem[BR]{BR}
Y.~Baryshnikov and D.~Romik,
Enumeration formulas for {Y}oung tableaux in a diagonal strip,
{\em Israel J.\ Math.} \textbf{178} (2010), 157--186.


\bibitem[Bia]{Bia}
P.~Biane, Representations of symmetric groups and free probability,
{\em Adv.\ Math.}~\textbf{138} (1998), 126--181.

\bibitem[Bri]{Bri}
G.~R.~Brightwell, The number of linear extensions of ranked
posets, LSE CDAM Res.\ Report~\textbf{18} (2003), 6~pp.

\bibitem[BT]{BT}
G.~R.~Brightwell and P.~Tetali,
The number of linear extensions of the Boolean lattice,
\emph{Order}~\textbf{20} (2003), 333--345.

\bibitem[BW]{BW}
G.~R.~Brightwell and P.~Winkler, Counting linear extensions,
\emph{Order}~\textbf{8} (1991), 225--242.

\bibitem[C+]{C+}
J.~Cardinal, S.~Fiorini, G.~Joret, R.~M.~Jungers and J.~I.~Munro,
Sorting under partial information (without the ellipsoid algorithm),
\emph{Combinatorica}~\textbf{33} (2013), 655--697.

\bibitem[CS]{CS}
L.~Carlitz and R.~Scoville,
Enumeration of rises and falls by position,
\emph{Discrete Math.}~\textbf{5} (1973), 45--59.

\bibitem[CGS]{CGS}
S.~Corteel, A.~Goupil and G.~Schaeffer,
Content evaluation and class symmetric functions,
\emph{Adv.~Math.}~\textbf{188} (2004), 315--336.

\bibitem[DVZ]{DVZ}
A.~Dembo, A.~Vershik and O.~Zeitouni,
Large deviations for integer partitions,
\emph{Markov Process.\ Related Fields}~\textbf{6} (2000), 147--179.

\bibitem[EKLP]{EKLP}
N.~Elkies, G.~Kuperberg, M.~Larsen and J.~Propp,
Alternating-sign matrices and domino tilings.~I,
\emph{J.~Algebraic Combin.}~\textbf{1} (1992), 111--132.


\bibitem[ERZ]{ERZ}
K.~Ewacha, I.~Rival and N.~Zaguia,
Approximating the number of linear extensions,
\emph{Theoret.\ Comput.\ Sci.}~\textbf{175} (1997), 271--282.

\bibitem[Feit]{Feit} W.~Feit,
The degree formula for the skew-representations of the symmetric group,
\emph{Proc.\ AMS}~\textbf{4} (1953), 740--744.

\bibitem[FeS]{FeS}
V.~F{\'e}ray and P.~{\'S}niady,
Asymptotics of characters of symmetric groups related to Stanley character formula,
\emph{Ann.\ of Math.}~\textbf{173} (2011), 887--906.

\bibitem[FlS]{FS}
P.~Flajolet and R.~Sedgewick, \emph{Analytic combinatorics},
Cambridge Univ.~Press, Cambridge, 2009.

\bibitem[FRT]{FRT}
J.~S. Frame, G.~de~B. Robinson and R.~M. Thrall,
The hook graphs of the symmetric group,
{\em Canad.\ J.\ Math.}~\textbf{6} (1954), 316--324.

\bibitem[GV]{GV}
I.~M.~Gessel and X.~G.~Viennot,
Determinants, paths and plane partitions, preprint (1989);
available at \. \href{http://tinyurl.com/zv3wvyh}{tinyurl.com/zv3wvyh}.


\bibitem[GM]{GM}
A.~Giambruno and S.~Mishchenko,
Degrees of irreducible characters of the symmetric group and exponential growth,
\emph{Proc.\ AMS}~\textbf{144} (2016), 943--953.

\bibitem[GJ]{GJ}
I.~P.~Goulden and D.~M.~Jackson,
\emph{Combinatorial enumeration}, John Wiley, New York, 1983.

\bibitem[GP]{GP}
R. Grigorchuk and I. Pak,
Groups of Intermediate Growth: an Introduction,
\emph{Ens.\ Math.}~\textbf{54} (2008), 251--272.


\bibitem[HG]{HG}
A.~M. Hamel and I.~P.~Goulden,
Planar decompositions of tableaux and {S}chur function determinants,
{\em Europ.\ J.~Combin.}~\textbf{16} (1995), 461--477.

\bibitem[HP]{HP}
A.~Hammett and B.~Pittel,  How often are two permutations comparable?
\emph{Trans.\ AMS}~\textbf{360} (2008), 4541--4568.


\bibitem[Ike]{Iken}
C.~Ikenmeyer,
\emph{Geometric Complexity Theory, Tensor Rank,
and Littlewood--Richardson Coefficients}, Ph.D.~thesis,
University of Paderborn, 2013, 213~pp.


\bibitem[KK]{KK}
J.~Kahn and J.~H.~Kim,
Entropy and sorting, \emph{J.~Comput.\ Syst.\ Sci.}~\textbf{51} (1995), 390--399.

\bibitem[KT]{KT}
A.~Knutson and T.~Tao,
Puzzles and (equivariant) cohomology of Grassmannians,
\emph{Duke Math.~J.}~\textbf{119} (2003), 221--260.

\bibitem[K1]{Kra}
C.~Krattenthaler,
Bijections for hook pair identities,
\emph{Electron.\ J.~Combin.}~\textbf{7} (2000),3
RP~27, 13~pp.

\bibitem[K2]{Kra-plane}
C.~Krattenthaler,
Plane partitions in the work of Richard Stanley and his school,
{\tt  arXiv:}{\tt 1503.05934}.

\bibitem[KS]{KS}
C.~Krattenthaler and M.~J.~Schlosser,
The major index generating function of standard Young tableaux of shapes of the form
``staircase minus rectangle'', in \emph{Ramanujan~125}, AMS, Providence, RI,
2014, 111--122.

\bibitem[Kre]{Kre}
V.~Kreiman,
{S}chubert classes in the equivariant {K}-theory and equivariant
  cohomology of the {G}rassmannian, {\tt arXiv:math.AG/0512204}.

\bibitem[LP]{LP}
A.~Lascoux and P.~Pragacz, Ribbon {S}chur functions,
{\em Europ.\ J.~Combin.}~\textbf{9} (1988), 561--574.

\bibitem[LasS]{LaS}
A.~Lascoux and M.-P. Sch\"utzenberger,
Polyn\^{o}mes de {S}chubert (in French),
{\em C.R.~Acad.\ Sci.\ Paris S\'er.~I Math.}~\textbf{294} (1982), 447--450.

\bibitem[LS]{LS}
B.~F.~Logan and L.~A.~Shepp,
A variational problem for random Young tableaux,
in \emph{Advances in Math.}~\textbf{26} (1977), 206--222.

\bibitem[MPP1]{MPP1}
A.~H.~Morales, I.~Pak and G.~Panova,
Hook formulas for skew shapes I. $q$-analogues and bijections; {\tt arXiv:} {\tt 1512.08348}.

\bibitem[MPP2]{MPP2}
A.~H.~Morales, I.~Pak and G.~Panova,
Hook formulas for skew shapes II. Combinatorial proofs and enumerative applications,
{\tt arXiv:1610.04744}; to appear in \emph{SIAM Jour.\ Discrete Math.}

\bibitem[MPP3]{MPP3}
A.~H.~Morales, I.~Pak and G.~Panova,
Hook formulas for skew shapes III. Multivariate formulas from Schubert calculus,
in preparation.

\bibitem[N1]{Na}
 H.~Narayanan,
Estimating deep Littlewood--Richardson coefficients, in
\emph{Proc.\ FPSAC~2014}, DMTCS, Nancy, 2014, 321--332

% \bibitem[Nara]{Nar}
% H.~Narayanan, On the complexity of computing {K}ostka numbers and
% {L}ittlewood-{R}ichardson coefficients,
% {\em J. Algebr. Comb.}~\textbf{24} (2006), 347--354.

\bibitem[N2]{Nar}
H.~Naruse, {S}chubert calculus and hook formula, talk slides at \emph{73rd
{S}\'em.\ {L}othar.\ {C}ombin.},
Strobl, Austria, 2014; available at \.
  \href{http://www.emis.de/journals/SLC/wpapers/s73vortrag/naruse.pdf}{tinyurl.com/z6paqzu}.


% \bibitem[NPS]{NPS}
% J.-C.~Novelli, I.~Pak and A.~V.~Stoyanovskii,
% A direct bijective proof of the hook-length formula,
% {\em Discrete Math.\ Theor.\ Comput.\ Sci.}~\textbf{1} (1997), 53--67.

\bibitem[OO]{OO}
A.~Okounkov and G.~Olshanski,
Shifted {S}chur functions,
{\em St.~Petersburg Math.~J.}~\textbf{9} (1998), 239--300.


% \bibitem[Pak]{P1}
% I.~Pak,
% Hook length formula and geometric combinatorics.
% {\em S\'em.\ Lothar.\ Combin.}~\textbf{46} (2001), Art.~B46f, 13~pp.

\bibitem[Pak]{P2}
I.~Pak, On the number of tableaux of large diagrams, in preparation.

\bibitem[Pan]{Pan}
G.~Panova, Tableaux and plane partitions of truncated shapes,
\emph{Adv.\ Appl.\ Math.}~\textbf{49} (2012), 196--217.

\bibitem[PR]{PR}
B.~Pittel and D.~Romik,
Limit shapes for random square Young tableaux,
\emph{Adv.\ Appl.\ Math.}~\textbf{38} (2007), 164--209.

\bibitem[Pro]{Pro}
R.~A.~Proctor, New symmetric plane partition identities from invariant theory work of De Concini
and Procesi, \emph{Europ.\ J.\ Combin.}~\textbf{11} (1990), 289--300.

\bibitem[RV]{RV}
A.~Regev and A.~Vershik,
Asymptotics of Young diagrams and hook numbers,
\emph{Electron.\ J.~Combin.}~\textbf{4} (1997),
no.~1, RP~22, 12~pp.


\bibitem[Rom]{Rom}
D.~Romik,
\emph{The surprising mathematics of longest increasing subsequences},
Cambridge Univ.~Press, New York, 2015.

\bibitem[Sag]{Sag}
B.~E. Sagan, {\em The Symmetric Group}, Springer, New York, 2001.

\bibitem[SK]{SK}
J.~Sha and D.~J.~Kleitman,  The number of linear extensions of
subset ordering, \emph{Discrete Math.}~\textbf{63} (1987), 271--278.

\bibitem[OEIS]{OEIS}
N.~J.~A.~Sloane,
The {O}nline {E}ncyclopedia of {I}nteger {S}equences,
\href{http://oeis.org}{\tt oeis.org}.


\bibitem[S1]{Stanley_sym}
R.~P.~Stanley,
Symmetries of plane partitions,
\emph{J.~Combin.\ Theory, Ser.~A}~\textbf{43} (1986), 103--113.

\bibitem[S2]{Stanley_skewSYT}
R.~P.~Stanley, On the enumeration of skew {Y}oung tableaux,
{\em Adv.\ Appl.\ Math.}~\textbf{30} (2003), 283--294.

\bibitem[S3]{EC2}
R.~P.~Stanley, {\em Enumerative Combinatorics}, vol.~1 (second ed.)
and vol.~2, Cambridge Univ.~Press, 2012 and~1999; supplementary excercies
available at \href{http://www-math.mit.edu/~rstan/ec}{http://www-math.mit.edu/\~{}rstan/ec}.


\bibitem[S4]{Stanley_ICM}
R.~P.~Stanley, Increasing and decreasing subsequences and their variants,
in \emph{Proc.~ICM}, Vol.~I, EMS, Z\"urich, 2007, 545--579.

\bibitem[S5]{Stanley_SurveyAP}
R.~P.~Stanley, A survey of alternating permutations, in
{\em Combinatorics and graphs}, AMS, Providence, RI, 2010, 165--196.


\bibitem[Tro]{Tro}
W.~T.~Trotter,
Partially ordered sets, in \emph{Handbook of combinatorics}, Vol.~1,
Elsevier, Amsterdam, 1995, 433--480.

\bibitem[VK]{VK}
A.~M.~Vershik and S.~V.~Kerov, The asymptotic character theory of the
symmetric group, \emph{Funct.\ Anal.\ Appl.}~\textbf{15} (1981), 246--255.

\bibitem[Wac]{Wac}
M.~Wachs, Flagged {S}chur functions, {S}chubert polynomials, and symmetrizing
  operators,  {\em J.~Combin.\ Theory, Ser.~A}~\textbf{40} (1985), 276--289.

\end{thebibliography}
\end{document}